\documentclass[12pt]{amsart}

\usepackage{amssymb,amsbsy}
\usepackage{latexsym,color,longtable}
\usepackage{verbatim,euscript}
\usepackage{fullpage,array}
\usepackage{tikz}
\usetikzlibrary{arrows,shapes,trees}
\usepackage[matrix,arrow,curve]{xy}

\numberwithin{equation}{section}
\usepackage[colorlinks=true,linkcolor=blue,urlcolor=violet,citecolor=magenta]{hyperref}

\input {cyracc.def}

\definecolor{my_color}{rgb}{0,0.5,0.5}
\definecolor{MIXT}{rgb}{0.4,0.3,0.6}
\definecolor{mixt}{rgb}{0.5,0.3,0.2}
\definecolor{sin}{rgb}{0,0.5,0.5}
\definecolor{darkblue}{rgb}{0,0.1,0.8}
\definecolor{redi}{rgb}{0.5,0,0.4}
\definecolor{gri}{rgb}{0.6,1,0}
\definecolor{darkgreen}{rgb}{0.09, 0.35, 0.27}
\definecolor{forest}{rgb}{0.13, 0.55, 0.13}

\catcode`\@=\active
\catcode`\@=11

\renewcommand{\@cite}[2]{[{{\bf #1}\if@tempswa , #2\fi}]}
\renewcommand{\@biblabel}[1]{[{\bf #1}]\hfill}
\catcode`\@=12

\newtheorem{thm}{Theorem}[section]
\newtheorem{lm}[thm]{Lemma}
\newtheorem{cl}[thm]{Corollary}
\newtheorem{prop}[thm]{Proposition}

\theoremstyle{remark}
\newtheorem{rmk}[thm]{Remark}

\theoremstyle{definition}
\newtheorem{ex}[thm]{Example}

\newenvironment{E6}[6]{%
{\footnotesize\begin{tabular}{@{}c@{}}
{#1}--{#2}--\lower3.55ex\vbox{\hbox{{#3}\rule{0ex}{2.5ex}}
\hbox{\hspace{0.4ex}\rule{.1ex}{1ex}\rule{0ex}{1.4ex}}\hbox{{#6}\strut}}--{#4}--{#5}
\end{tabular}}}

\newcommand {\be}{{\mathfrak b}}

\newcommand {\gD}{{\mathfrak D}}

\newcommand {\G}{{\mathfrak G}}

\newcommand {\g}{{\mathfrak g}}

\newcommand {\ka}{{\mathfrak k}}
\newcommand {\el}{{\mathfrak l}}

\newcommand {\n}{{\mathfrak n}}

\newcommand {\p}{{\mathfrak p}}
\newcommand {\q}{{\mathfrak q}}

\newcommand {\es}{{\mathfrak s}}
\newcommand {\te}{{\mathfrak t}}
\newcommand {\ut}{{\mathfrak u}}
\newcommand {\fX}{{\mathfrak X}}
\newcommand {\z}{{\mathfrak z}}

\newcommand {\sln}{{\mathfrak {sl}}_n}
\newcommand {\slN}{{\mathfrak {sl}}_N}

\newcommand {\glN}{{\mathfrak {gl}}_N}
\newcommand {\glv}{{\mathfrak {gl}}(\eus V)}
\newcommand {\glbv}{{\mathfrak {gl}}(\BV)}

\newcommand {\spv}{{\mathfrak {sp}}(\VV)}
\newcommand {\spn}{{\mathfrak {sp}}_{2n}}

\newcommand {\sono}{{\mathfrak {so}}_{2n+1}}
\newcommand {\sone}{{\mathfrak {so}}_{2n}}
\newcommand {\son}{{\mathfrak {so}}_{n}}
\newcommand {\soN}{{\mathfrak {so}}_{N}}

\newcommand {\eus}{\EuScript}

\newcommand {\VV}{\eus V}

\newcommand {\esi}{\varepsilon}
\newcommand {\ap}{\alpha}

\newcommand {\lb}{\lambda}
\newcommand {\vp}{\varphi}
\newcommand {\vth}{\vartheta}
\newcommand {\bvth}{\bar\vartheta}
\newcommand {\blb}{{\boldsymbol{\lambda}}}
\newcommand {\bmu}{{\boldsymbol{\mu}}}

\newcommand {\tvp}{\tilde{\varphi}}

\newcommand {\cD}{{\mathcal D}}

\newcommand {\cI}{{\eus I}}
\newcommand {\cJ}{{\eus J}}

\newcommand {\N}{{\mathfrak N}}
\newcommand {\co}{{\mathcal O}}

\newcommand {\sD}{{\mathsf D}}
\newcommand {\BV}{{\mathbb{V}}}

\newcommand {\BZ}{{\mathbb Z}}
\newcommand {\BP}{{\mathbb P}}
\newcommand {\BT}{{\mathbb T}}

\newcommand {\md}{/\!\!/}

\newcommand {\ad}{{\mathrm{ad\,}}}

\newcommand {\Ann}{{\mathrm{Ann\,}}}
\newcommand {\codim}{{\mathrm{codim\,}}}
\newcommand {\codims}{{\mathrm{codim}}}
\newcommand {\cha}{{\mathsf{char}}}

\newcommand {\Lie}{{\mathrm{Lie\,}}}

\newcommand {\Ima}{{\mathsf{Im}}}
\newcommand {\rank}{{\mathrm{rank}}}
\newcommand {\rk}{{\mathsf{rk}}}
\newcommand {\spe}{{\mathsf{Spec\,}}}
\newcommand {\sym}{\mathsf{Sym}}
\newcommand {\sk}{\mathsf{Sk}}

\newcommand {\trdeg}{{\mathrm{trdeg\,}}}

\newcommand {\tri}{{\mathfrak{sl}}_2}
\newcommand {\sltri}{{\mathfrak{sl}}_3}
\newcommand {\GR}[2]{{\textrm{{\sf\bfseries #1}}}_{#2}}

\newcommand {\ov}{\overline}
\newcommand {\un}{\underline}

\newcommand {\beq}{\begin{equation}}
\newcommand {\eeq}{\end{equation}}

\newcommand{\curle}{\preccurlyeq}
\renewcommand{\le}{\leqslant}
\renewcommand{\ge}{\geqslant}
\renewcommand{\lg}{\langle}
\newcommand{\rg}{\rangle}
\newcommand{\bbk}{\Bbbk}

\newcommand {\omin}{\co_{\sf min}}
\newcommand {\osp}{\co_{\sf sph}}
\newcommand {\bco}{\ov{\co}}

\newcommand {\Sec}{\boldsymbol{\sigma}}
\newcommand {\Tan}{\boldsymbol{\tau}}
\newcommand {\CS}{\textrm{{\sl\bfseries CS}}}
\newcommand {\CJ}{\textrm{{\sl\bfseries CJ}}}

\newcommand {\ups}{\Upsilon}
\newcommand {\ess}{{\es_\star}}
\newcommand {\els}{{\el_\star}}

\newcommand{\odin}{\mathrm{{1}\!\! 1}}

\begin{document}
\setlength{\parskip}{2pt plus 4pt minus 0pt}
\hfill {\scriptsize December 30, 2024} 
\vskip1ex

\title[Secant varieties]{Nilpotent orbits and their secant varieties}
\author{Dmitri I. Panyushev}
\address{Higher School of Modern Mathematics MIPT,
Moscow 115184, Russia}
\email{panyush@mccme.ru}
\thanks{}
\begin{abstract}
Let $G$ be a simple algebraic group, $\g=\Lie G$, and $\co$ a nilpotent orbit in $\g$. Let $\CS(\co)$ 
denote the affine cone over the secant variety of $\ov{\BP\co}\subset \BP\g$. Using the theory of doubled
$G$-actions, we describe $\CS(\co)$ for all $\co$. Let $c(\co)$ and $r(\co)$ denote the {\it complexity} 
and {\it rank} of the $G$-variety $\co$. It is proved that $\dim\CS(\co)=2\dim\co-2c(\co)-r(\co)$ and there 
is a subspace $\te_\co$ of a Cartan subalgebra of $\g$ such that $\CS(\co)$ is the closure of 
$G{\cdot}\te_\co$. We compute $c(\co)$ and $r(\co)$ for all nilpotent orbits and show that $\CS(\co)$ is 
the closure of $\Ima(\mu)$, where $\mu: \BT^*(\co)\to \g^*\simeq\g$ is the moment map. It is also shown
that the secant variety of $\ov{\BP\co}$ is defective if and only if $r(\co)<\rk\,G$ if and only if 
$\CS(\co)\ne \g$. 
\end{abstract}
\maketitle

\tableofcontents
\section{Introduction}

\noindent
The ground field $\bbk$ is algebraically closed and $\cha(\bbk)=0$. Let $G$ be a simple algebraic 
group with $\Lie G=\g$ and $\N=\N(\g)$ the nilpotent cone in $\g$. Then $\N/G$ stands for the finite set of 
nilpotent $G$-orbits. Each $\co\in\N/G$ is conical, and $\BP\co\subset\BP\g$ is its projectivisation.
 
For $\co\in\N/G$, let $\Sec(\ov{\BP\co})$ be the {\it secant variety\/} of the projective variety 
$\ov{\BP\co}$. If $\dim\Sec(\ov{\BP\co})$ is less than the {\it expected dimension}, 
$\min\{\dim\BP\g, 2\dim \ov{\BP\co}+1\}$, then $\Sec(\ov{\BP\co})$ is said to be {\it defective}.
The affine cone over $\Sec(\ov{\BP\co})$ in $\g$ is said to be the {\it conical secant 
variety\/} of $\co$, denoted $\CS(\co)$.
In~\cite{omoda}, Omoda describes $\CS(\co)$ for most nilpotent orbits in the
classical Lie algebras via a case-by-case analysis. However, his results for $\g=\son$ are 
incomplete. In this paper, a general approach to $\CS(\co)$ is developed, which exploits the doubled 
$G$-action associated with $\co$ and the notions of complexity and rank of $G$-varieties~\cite{p99}. 
Our main results do not invoke the classification of nilpotent orbits or simple Lie algebras. 

We fix a Borel subgroup $B\subset G$ and a maximal torus $T\subset B$. Then $U=(B,B)$ is the fixed 
maximal unipotent subgroup of $G$. 
Let $c(X)$ (resp. $r(X)$) denote the {\it complexity\/} (resp. {\it rank}) of an irreducible $G$-variety $X$, 
see Section~\ref{sect:2} for details. Then $r(X)\le \rk\, G$. We prove that, for any $\co\in\N/G$, 
\[
        \dim\CS(\co)=2\dim\co-2c(\co)-r(\co) .
\]
This implies that $\Sec(\ov{\BP\co})$ is defective if and only if $\CS(\co)\ne\g$ if and only if 
$r(\co)<\rk\,G$. For each $\co$, we decribe a subalgebra $\te_\co\subset\te=\Lie T$ such that 
$\dim\te_\co=r(\co)$ and
\beq           \label{eq:CS} 
   \CS(\co)=\ov{G{\cdot}\te_\co} . 
\eeq
Our constructions show that the {\it secant\/} and {\it tangential\/} varieties of $\ov{\BP\co}$ coincide. Yet 
another interpretation of $\CS(\co)$ exploits the cotangent bundle of $\co$, $\BT^*(\co)$. Let   
$\mu: \BT^*(\co)\to \g^*\simeq\g$  be the moment map. Then $\CS(\co)$ is the closure of $\Ima(\mu)$.

The proofs rely on the theory of ``doubled actions'' of reductive groups and structure results on related 
generic stabilisers~\cite{p90,p99}. For any $G$-variety $X$, one defines the dual $G$-variety $X^*$. 
Here $X\simeq X^*$ as a variety, but the $G$-actions on $X^*$ is obtained via a twist, see 
Section~\ref{subs:twist}. The diagonal $G$-action on $X\times X^*$ is said to be {\it doubled\/} (with 
respect to the given $G$-action on $X$). If $X$ is irreducible, then the doubled action always has a 
generic stabiliser, say $S_\star$, which is reductive. Furthermore, there is a 
Levi subgroup $L_\star\subset G$ such that $(L_\star,L_\star)\subset S_\star\subset L_\star$. In terms of Lie algebras, this yields 
the decomposition $\el_\star=\es_\star\oplus \te_X$, where $\te_X$ is orthogonal to $\es_\star=\Lie S_\star$. 
For $X=\co\in\N/G$, this construction provides the subalgebra $\te_\co$ occurring in \eqref{eq:CS}. 
If $\es_\star$ is semisimple, then $\es_\star=[\el_\star,\el_\star]$ and $\te_\co$ is the centre of the Levi 
subalgebra $\el_\star$. Therefore, in these cases, $\CS(\co)$ is the closure of the {\it Dixmier sheet\/} $\gD(\el_\star)$ in 
$\g$ associated with $\el_\star$. Moreover, for any $\es_\star$, the dense orbit in $\CS(\co)\cap\N$ is the nilpotent orbit in $\gD(\el_\star)$. In particular, this dense orbit is always Richardson.

Since $\co$ is quasiaffine, there is a more direct construction of $\te_\co$~\cite{p99}. Let 
$\fX_+\subset\te^*$ be the monoid of dominant weights with respect to $(B,T)$ and
$\bbk[\co]=\bigoplus_{\lb\in\fX_+}\bbk[\co]_{\lb}$ the 
{\it isotypic decomposition}~\cite[Ch.\,II, \S\,3.1]{kr84}. Let 
$\Gamma_\co:=\{\lb\in \fX_+\mid \bbk[\co]_{\lb}\ne 0\}$ be the {\it weight} (aka {\it rank}) {\it monoid} of 
$\co$ and let $\lg\Gamma_\co\rg\subset \te^*$ denote its $\bbk$-linear span. Then 
$\dim \lg\Gamma_\co\rg=r(\co)$ and,
upon the usual identification of $\te$ and $\te^*$, we have $\te_\co=\lg\Gamma_\co\rg$. 

Results of~\cite{omoda} show that, for most nilpotent orbits $\co=G{\cdot}e$ in the {\sl classical\/} Lie 
algebras, $\CS(\co)$ depends only on the rank of matrix $e$ in the tautological representation of $\g$. 
This suggests to consider the partition of the poset $\N/G$ into a disjoint union of certain posets 
$\{\ups_\xi\}_{\xi\in \Xi}$, which are defined via partitions. We prove that, for each $\ups_\xi$, the generic 
isotropy subalgebra $\ess$ is the same for all $\co\in \ups_\xi$. The proof relies on computations 
with $\tri$-triples and $\BZ$-gradings related to the minimal and maximal elements of the posets 
$\{\ups_\xi\}_{\xi\in\Xi}$. This also implies that the map $\co\mapsto\CS(\co)$ is constant on each 
$\ups_\xi$. As a by-product, we obtain the values of $c(\co)$ and $r(\co)$ for all nilpotent orbits, which 
complements results for the spherical nilpotent orbits in~\cite{p94}. Then we describe the canonical embedding of $\ess$ in $\g$, which allows us to explicitly determine $\lg\Gamma_\co\rg$ and $\te_\co$,
and point the dense orbit in $\CS(\co)\cap\N$.

The paper is organised as follows. 
Section~\ref{sect:2} contains preliminaries on doubled actions, nilpotent orbits, and secant varieties. 
The main structure results on $\CS(\co)$ for $\co\in\N/G$ are presented in Section~\ref{sect:main}. In 
Sections~\ref{sect:4}, we recall results of~\cite{omoda}, with some complements for $\g=\soN$, and 
determine the maximal orbits with defective secant varieties. In Section~\ref{sect:comput-c-r}, we deal 
with the posets $\{\ups_\xi\}$ and related properties of $\CS(\co)$ and $\ess$ for the classical Lie 
algebras. Explicit results for the exceptional Lie algebras are gathered in Section~\ref{sect:defective-exc}. 
Section~\ref{sect:data} contains explicit data on the canonical embedding of $\ess$ and the dense orbit in 
$\CS(\co)\cap\N$. In Section~\ref{sect:appl},  we give a curious property of posets $\{\ups_\xi\}$, point out 
certain constraints on possible values of $c(\co)$ and $r(\co)$, and briefly discuss higher secant varieties of nilpotent orbits.

\noindent
\un{Main notation}. 
Throughout, $G$ is a connected reductive algebraic group and $\g=\Lie G$. 
\begin{itemize}
\item[--] $G^x$ is the isotropy group of $x\in\g$ in $G$ and $\g^x=\Lie G^x$ is the centraliser of $x$ in $\g$;
\item[--]  \ $\Delta$ is the root system of $\g$ with respect to $T$;
\item[--] \  $\Pi=\{\ap_1,\dots,\ap_n\}\subset \Delta$ is the set of simple roots corresponding to $B$ \ ($n=\rk\,\g$);
\item[--] \ $\varpi_1,\dots,\varpi_n$ \ are the fundamental weights corresponding to $\Pi$;
\item[--] \ $Q^o$ is the identity component of an algebraic group $Q$;
\item[--] \ If $Q$ acts on a variety $X$, then we also write $(Q:X)$.
\end{itemize}
\noindent
Algebraic groups are denoted by capital Roman letters and their Lie algebras are denoted by the 
corresponding small Gothic letters, e.g. $\Lie B=\be$, $\Lie S=\es$, and $\Lie U=\ut$. 
Our main references for invariant theory are~\cite{kr84, vp}. We refer to~\cite{CM} for terminology 
and basic results on nilpotent orbits in $\g$.

\section{Doubled actions, nilpotent orbits, and secant varieties}    
\label{sect:2}

\noindent
Let $\bbk[X]$ denote the algebra of regular functions on a variety $X$. If $X$ is irreducible, then 
$\bbk(X)$ is the field of rational functions on $X$. If $X$ is acted upon by an algebraic group $Q$, then 
$\bbk[X]^Q$ and $\bbk(X)^Q$ are the subalgebra and subfield of invariant functions, respectively. 
For $x\in X$, let $Q^x$ denote the {\it stabiliser} (or {\it isotropy group}) of $x$ in $Q$. Then 
$\q^x:=\Lie Q^x$. A stabiliser $Q^x$ is said to be {\it generic}, if there is a dense open subset 
$\Omega\subset X$ such that $x\in\Omega$ and $Q^y$ is $Q$-conjugate to $Q^x$ for all $y\in\Omega$. 
Then any $y\in\Omega$ is said to be {\it generic}. Whenever it is necessary to keep track of the group,
we refer to  $Q$-{\it generic points}. A Lie algebra $\q^y$ $(y\in\Omega)$ is said to be
a {\it generic isotropy subalgebra} (={\sf g.i.s.}).
If $Q$ is reductive and $X$ is affine and smooth, then generic stabilisers always exist. But this is not true 
in general~\cite[\S\,7]{vp}.

Let $X$ be an irreducible $G$-variety. The {\it complexity} of $X$ is 
$c_G(X) =\min_{x\in X}\codim_X B{\cdot}x$. To define the {\it rank} of $X$, $r_G(X)$, one considers the 
multiplicative group $\bbk(X)^{(B)}$ of all $B$-semiinvariants in $\bbk(X)$. If $F\in\bbk(X)^{(B)}$, then 
$b{\cdot}F=\lb_F(b)F$ for all $b\in B$ and a character $\lb_F\in \fX(B)\simeq \fX(T)\simeq \BZ^{\rk\g}$. 
Then $\G(X):=\{\lb_F\mid F\in\bbk(X)^{(B)}\}$ is a free abelian group and $r_G(X)$ is defined to be the 
rank of $\G(X)$. 
One can prove that $r_G(X) =\trdeg \bbk(X)^U-\trdeg \bbk(X)^B$~\cite[Theorem\,1.2.9]{p99}.

By the Rosenlicht theorem (see e.g. \cite[\S\,2.3]{vp}), we also have
\beq   \label{eq:Rosen}
    \text{$c_G(X)=\trdeg\, k(X)^{B}$ and $c_G(X)+r_G(X)=\trdeg\, k(X)^{U}$}.
\eeq
As usual, $X$ is said to be $G$-{\it spherical}, if $B$ has a dense orbit in $X$, i.e., $c_G(X)=0$. If the 
reductive group is clear from the context, then we merely write $c(X)$ and $r(X)$.

If a reductive group $G$ acts on an affine variety $Z$, then $\bbk[Z]^G$ is finitely-generated and the 
affine variety $Z\md G:=\spe\bbk[Z]^G$ is the {\it categorical quotient}, see~\cite[\S\,4]{vp}. The 
inclusion $\bbk[Z]^G\subset\bbk[Z]$ yields the {\it quotient morphism\/} $\pi_Z:Z\to Z\md G$, which is 
onto. Each fibre $\pi^{-1}_Z(\xi)$, $\xi\in Z\md G$, contains a unique closed orbit, i.e., $Z\md G$ 
parametrises the closed $G$-orbits in $Z$. 

\subsection{Twisted and doubled actions~\cite{p90,p99}}
\label{subs:twist}
Let  $\vth\in \mathsf{Aut}(G)$ be an involution of {\it maximal rank}, i.e., there is a maximal torus 
$T\subset G$ such that $\vth(t)=t^{-1}$ for all $t\in T$. 
Without loss of generality, we may assume that $T$ is our fixed maximal torus. Then $T=\vth(B)\cap B$.
Let $X$ be a $G$-variety with the action $(g,x)\mapsto g{\cdot}x$, where $g\in G$ and $x\in X$. The 
{\it twisted\/} action on $X$ is defined by 
$(g,x)\mapsto g\ast x=\vth(g){\cdot}x$. The variety $X$ equipped with the twisted $G$-action is said to be {\it dual\/} to 
$X$ and denoted by $X^*$. This terminology is justified by the fact that if $X=\VV$ is a $G$-module, then
$\VV^*$ is the dual $G$-module in the usual sense.

For the given action $(G:X)$, the {\it doubled action\/} is the diagonal $G$-action on $X\times X^*$. 
In other words, $g{\cdot}(x,y)=(g{\cdot}x,\vth(g){\cdot}y)$ for all $(x,y)\subset X\times X^*$. Note that the
usual {diagonal} $\Delta_X=\{(x,x\mid x\in X)\}\subset X\times X^*$ is {\bf not} $G$-stable w.r.t. the 
doubled action. The theory of doubled actions provides relations between various generic stabilisers and 
their connections to the rank and complexity of $X$. An element $t\in\te$ is said to be {\it dominant}, 
if $\ap(t)\ge 0$ for all $\ap\in\Pi$.
Then the following holds (cf.~\cite[Theorem\,1.2.2]{p99}):

The actions $(G:X\times X^*)$, $(B:X)$,  and $(U:X)$ have a generic stabiliser, and there is a $G$-generic point $z=(x,x)\in \Delta_X\subset X\times X^*$ such that

\begin{itemize}
\item[($\eus P_1$)]  $x\in X$ is also $B$-generic and $U$-generic;
\item[($\eus P_2$)]  for $S_\star:=G^z=G^x\cap \vth(G^x)$, there is a dominant $t\in \te$ such that 
$(G^t,G^t)\subset S_\star \subset G^t$;
\item[($\eus P_3$)]   $U_\star:=S_\star\cap U=U^x$ and $B_\star:=S_\star\cap B=B^x$. 
\end{itemize}
The generic stabiliser $S_\star$ for the doubled action $(G:\co\times\co^*)$ satisfying 
$(\eus P_1)$--$(\eus P_3)$ is said to be {\it canonical\/} w.r.t.~the pair $(B,T)$. We also say that this yields the {\it canonical embedding} $S_\star \hookrightarrow G$ (or $\ess\hookrightarrow \g$).
It follows from $(\eus P_2)$ and $(\eus P_3)$ that 
\begin{itemize}
\item[($\eus P_4$)] 
$S_\star$ is reductive, $U_\star$ is a maximal 
unipotent subgroup of $S$,  $(S_\star\cap T)^o$ is a maximal torus in $S_\star^o$,  and $B_\star^o$ is a 
Borel subgroup of $S_\star^o$. 
\item[($\eus P_5$)]  $B{\cdot}S_\star=P$ is a parabolic subgroup of $G$, $P\cap G^x=S_\star$, and
$P\cap\vth(P)$ is a standard Levi subgroup of $P$.
\end{itemize}
By $(\eus P_1)$ and \eqref{eq:Rosen}, one has
\[ 
 c(X)+r(X)=\dim X-\dim U+\dim U^x   \ \text{ and } \  c(X)=\dim X-\dim B+\dim B^x.
\] 
Taking either the sum or the difference of these equalities and using ($\eus P_4$), we obtain
$2c(X)+r(X)=2\dim X-\dim G+\dim S_\star$ \ and \ $r(X)=\rk\, G-\rk\, S_\star$. This implies that 
\beq     \label{eq:max-dim-orb}
      \max_{z\in X\times X^*}\dim G{\cdot}z=\dim G-\dim S_\star=2\dim X-2c(X)-r(X) .
\eeq
Hence $r(X)=\rk\,G$ if and only if $S_\star$ is finite, and then $c(X)=\dim X-\dim B$.

Let $\bbk[X]=\bigoplus_{\lb\in\fX_+}\bbk[X]_{\lb}$ be the {\it isotypic decomposition}~\cite[Ch.\,II, \S\,3.1]{kr84}. Then $\Gamma_X=\{\lb\in \fX_+\mid \bbk[X]_{\lb}\ne 0\}$ is the {\it weight monoid\/} of $X$. If
$X$ is quasiaffine, then $\G(X)$ is the group  generated by $\Gamma_X$~\cite[1.3.10]{p99}. Moreover, 
$\Gamma_X$ determines 
the canonical generic stabiliser $S_\star$ (or $B_\star$). That is, for the group 
$B_\star=T_\star{\cdot}U_\star$ and $\ut_\star=\Lie U_\star$, one has
\begin{gather}  \label{eq:t-ast}
\text{$T_\star=\{t\in T \mid \gamma(t)=1 \ \ \forall \gamma\in \Gamma_X\}$ \qquad \cite[1.2.9]{p99}} \\
\label{eq:u-ast}
\text{$\displaystyle \ut_\star=\bigoplus_{\beta\in\Delta^+, \ (\beta, \Gamma_X)=0} \g_\beta$ \qquad \qquad \cite[1.3.12]{p99}}
\end{gather}
Conversely, these formulae show that the group $\G(X)$ is determined by the canonical stabiliser
$S_\star$. 
It follows from~\eqref{eq:t-ast} that $r(X)>0$ for the quasiaffine varieties. Translating~\eqref{eq:t-ast} into
the Lie algebra setting shows that, for $\te_X=\lg\Gamma_X\rg\subset\te^*\simeq\te$, one has
$\te=\te_X\oplus\te_\star$ and $(\te_X)^\perp=\te_\star$. Then using~\eqref{eq:u-ast}, we conclude
that the centraliser of $\te_X$ in $\g$ equals $\es_\star\oplus\te_X=\el_\star$.

\subsection{Nilpotent orbits}    
\label{subs:nilp}
We regard $\N/G$ as poset {w.r.t.}~the closure relation, i.e., $\co_1\curle\co_2$ if and only if
$\co_1\subset\ov{\co_2}$. Then the regular nilpotent orbit, $\co_{\sf reg}$, is the maximal element of
$\N/G$.
For the classical Lie algebras, the structure of the poset $(\N/G,\curle)$ is determined 
by partitions~\cite[Ch.\,6.2]{CM}. If $\g$ is exceptional, then the Hasse diagrams of $(\N/G,\curle)$ can 
be found in~\cite[IV.2]{spalt}.

For $e\in\N\setminus\{0\}$, let $\{e,h,f\}$ be an $\tri$-triple in $\g$, i.e., $[h,e]=2e$, $[e,f]=h$, and 
$[h,f]=-2f$~\cite[Ch.\,6,\,\S2]{t41}. The semisimple element $h$ is called a {\it characteristic\/} of $e$. All 
$\tri$-triples with a given $h$ are $G^h$-conjugate and all $\tri$-triples with a given $e$ are 
$G^e$-conjugate, see~\cite[3.4]{CM}.
Without loss of generality, one may assume that $h\in\te$ and $\ap(h)\ge 0$ for all $\ap\in\Pi$. Then $h$ 
is said to be the {\it dominant characteristic} (w.r.t. chosen $\te$ and $\Delta^+$). By a celebrated result 
of {Dynkin} (1952), one then has $\ap(h)\in\{0,1,2\}$~\cite[Ch.\,6,\,\S\,2,\,Prop\,2.2]{t41}. 
The {\it weighted Dynkin diagram\/} of $\co=G{\cdot}e\subset\N$, denoted
$\eus D(\co)$ or $\eus D(e)$, is the Dynkin diagram of $\g$ equipped with numerical marks 
$\{\ap(h)\}_{\ap\in\Pi}$. 

Let $\g=\bigoplus_{i\in\BZ}\g(i)$ be the $\BZ$-grading determined by $h$, i.e.,  
$\g(j)=\{v\in \g \mid [h,v]=jv\}$. It will be referred to as the $(\BZ,h)$-{\it grading\/} of $\g$. Then 
$\p=\g({\ge}0)=\bigoplus_{i\ge 0}\g(i)$ is the parabolic subalgebra associated with $e$, $\el:=\g(0)$ is a 
Levi subalgebra of $\p$, and $e\in\g(2)$. Write $L\subset P$ for the associated connected subgroups of 
$G$. Clearly, each $\g(i)$ is an $L$-module. It is known that $G^e\subset P$, $L{\cdot}e$ is dense in
$\g(2)$, and $K:=L^e=L\cap G^e$ is a maximal reductive subgroup of $G^e$. The nilpotent orbit 
$G{\cdot}e$ is said to be {\it distinguished}, if $K$ is finite, i.e., $\dim\g(0)=\dim\g(2)$. The orbit 
$G{\cdot}e$ is {\it even}, if $\g(1)=\{0\}$, i.e., $\ap(h)\in\{0,2\}$ for all $\ap\in\Pi$.

The $K$-module $\g(2i)$ is orthogonal for each $i\in\BZ$~\cite[Prop.\,1.2]{aif99}. 
Let $S\subset K$ be a generic stabiliser for the $K$-module $\g(2)$.
Then $S$ is reductive and the complexity and rank of $\co=G{\cdot}e$ can be computed in terms of the 
$(\BZ,h)$-grading as follows.

\begin{thm}[{\cite[Theorem\,2.3]{p94}}]   \label{thm:p94}
For any nontrivial nilpotent orbit, we have
\[   c_G(\co)=c_L(L/K)+c_S\bigl(\g({\ge}3)\bigr) \ \text{ and } \ 
     r_G(\co)=r_L(L/K)+  r_S\bigl(\g({\ge}3)\bigr) .
\]    
\end{thm}
\noindent
This result reduces computing the complexity and rank of nilpotent orbits to the similar task for 
{\bf affine} homogeneous spaces and {\bf representations} of reductive groups.
Here $S$ is a generic stabiliser for the doubled action $(L: (L/K)\times (L/K)^*)$. Hence
$r_L(L/K)=\rk\,L-\rk\,S=\rk\,G-\rk\,S$ and $c_L(L/K)=\dim L+\dim S-2\dim K$.
There are also useful complements to this result. 
\\ \indent (1) \ a generic stabiliser $S_\star\subset S$ for the doubled action $(S: \g({\ge}3)\times \g({\ge}3)^*)$ is 
also a generic stabiliser for $(G: \co\times\co^*)$. Moreover, if $S$  is chosen to be canonical 
w.r.t.~$(B\cap L, T)$ and $S_\star$ is canonical w.r.t.~$(B\cap S,T\cap S)$, then 
$S_\star$ is canonical for $(G:\co\times\co^*)$ w.r.t.~$(B,T)$, see~\cite[Theorem\,1.2(iii)]{p94}.
\\ \indent
(2) \ $r_G(\co)=r_L \bigl( \g({\ge}2)\bigr)$ and $c_G(\co)=c_L \bigl(\g({\ge}2)\bigr)$~\cite[Theorem\,4.2]{aif99};

In explicit examples, the nilpotent orbits in the classical Lie algebras $\slN,\soN,\spn$ are represented by
partitions $\blb$ of $N$ or $2n$. If $\co_\blb\subset\N$ corresponds to $\blb$, then
there are explicit formulae for $\dim\co_\blb$, $\eus D(\co_\blb)$, and $K$ via $\blb$, see
Chapters~5,6 in~\cite{CM}. For the orbits in the exceptional Lie algebras, we use the standard labels, see 
Tables in~\cite[Ch.\,8.4]{CM}.

\begin{rmk}   \label{rem:exc-fonts}
To distinguish the exceptional algebraic groups,  their Lie algebras, and labels of their nilpotent orbits,
we use {\sl different fonts}. For instance, $\GR{E}{7}$ is the exceptional group, $\eus E_7$ is its Lie algebra,
and, say,  $\mathsf E_6(a_1)$ and $2{\mathsf A}_1$ are nilpotent $\GR{E}{7}$-orbits in $\eus E_7$.
\end{rmk}

\subsection{Secant and tangential varieties}    
\label{subs:secant}
Let $X\subset \BP^n=\BP\eus V$ be a projective variety. The {\it secant variety\/} of $X$, $\Sec(X)$,
is  the closure of the union of all secant lines to $X$. That is, if $x,y\in X$ are different points and 
$\BP^1_{xy}\subset \BP(\eus V)$ is the line through $x$ and $y$, then
\[
    \Sec(X)=\ov{\bigcup_{x,y\in X} \BP^1_{x,y}} ,
\]
see e.g.~\cite[Chapter\,5]{land}. If $X$ is irreducible, then so is $\Sec(X)$, and the expected dimension of 
$\Sec(X)$ is $\min\{2\dim X+1, n\}$. If $\dim\Sec(X) < \min\{2\dim X+1, n\}$,  
then $\Sec(X)$ is said to be {\it defective}. Then  $\delta_X=2\dim X+1-\dim\Sec(X)$ is
called  the {\it defect\/}  of $\Sec(X)$ or the {\it secant defect}\/ of $X$.

The {\it tangential variety\/} of $X$, denoted $\Tan(X)$, is the closure of
$\bigcup_x T_{x, X}\subset \BP^n$, where $p$ runs over the smooth points of $X$ and $T_{x,X}$ is
the embedded projective tangent space of $X$ at $x$, cf.~\cite[Chapter\,8.1]{land}.
Then $\dim\Tan(X)\le \min\{2\dim X, n\}$ and $\Tan(X)\subset\Sec(X)$.

Let $\widehat X\subset \eus V$ denote the affine cone over $X$. That is, if 
$p:\eus V\setminus\{0\}\to \BP(\eus V)$ is the canonical map, then $\widehat X=p^{-1}(X)\cup\{0\}$. 
We say $\widehat {\Sec(X)}$ is the {\it conical secant variety} of $X$ (or $\widehat X$). 
In terms of affine varieties, $\Sec(X)$ is defective, if 
$\dim\widehat {\Sec(X)}<\min\{\dim\eus V, 2\dim\widehat X\}$. Then
the formula for the secant defect of $X$ reads 
\beq          \label{eq:aff-defect}
    \delta_X=2\dim\widehat X- \dim \widehat {\Sec(X)}.
\eeq
Likewise, $\widehat{\Tan(X)}\subset\eus V$ is the {\it conical tangential variety} of $X$.

We only deal with these varieties if $\eus V=\g$, $\co\in\N/G$, and $X=\ov{\BP\co}$. Then 
$\widehat{\ov{\BP\co}}=\bco$ and we write $\CS(\co)$ in place of $\widehat{\Sec(\ov{\BP\co})}$.
Abusing the language, we also say that that $\co$ is {\it defective}, if the secant variety 
$\Sec(\ov{\BP\co})\subset \BP\g$ is defective.

\section{The secant variety of a nilpotent orbit} 
\label{sect:main}
\noindent
From now on, $G$ is assumed to be simple.
Recall that $\vth$ is an involution of maximal rank of $G$. Let $\bar\vth$ denote
the induced involution of $\g$. Then we have $\bar\vth(g{\cdot}x)=\vth(g){\cdot}\bar{\vth}(x)$ for all
$g\in G$ and $x\in \g$.

\begin{lm}[{\cite[Lemma\,2.10]{ls79}}]   \label{lm:theta-stab}
Let $\co$ be a nilpotent orbit in $\g$. Then $\bar\vth(\co)=\co$.
\end{lm}

Consider the doubled action $(G:\co\times\co^*)$ associated with a nilpotent orbit $\co$.
\begin{lm}    \label{lm:doubled-diag}
The doubled action  $(G:\co\times\co^*)$ is equivalent to the diagonal action $(G:\co\times\co)$.
\end{lm}
\begin{proof}
Since $\bar\vth(\co)=\co$, the morphism $\psi: \co\times\co^*\to \co\times\co$,
$(x,y)\stackrel{\psi}{\mapsto} (x,\bar\vth(y))$, provides a required equivalence.
\end{proof}

Therefore, dealing with generic points and generic stabilisers, we may consider the usual diagonal
$G$-action on  $\co\times\co$ in place of the doubled action related to $(G:\co)$. In this case,
the diagonal $\Delta_\co\subset \co\times\co^*$ should be replaced with the {\it twisted diagonal}
\beq            \label{eq:theta-diag}
    \Delta_{\co,\bvth}=\{x,\bar{\vth}(x)\mid x\in \co\}\subset \co\times\co .
\eeq
Translating properties $(\eus P_1)$--$(\eus P_3)$ into this setting, one concludes that there is a 
$G$-generic point $(x,\bar\vth(x))\in \co\times\co$ such that $S_\star=G^x\cap G^{\bar\vth(x)}$ is the canonical generic stabiliser and
$x$ is both $B$-generic and $U$-generic.
\\ 
\indent Let $\BT_x(\co)\subset \g$ be the embedded tangent space at $x\in\co$.

\begin{thm}    \label{thm:main-dim}
For $\co\in \N/G$, let $S_\star$ be a generic stabiliser for $(G:\co\times\co)$. Then
\[ 
     \dim\CS(\co)=2\dim\co-2c(\co)-r(\co)=\dim G-\dim S_\star . 
\] 
\end{thm}
\begin{proof}
By Terracini's Lemma~\cite[Chap.\,5.3]{land}, 
$\dim\CS(\co)=\dim \bigl(\BT_x(\co)+\BT_y(\co)\bigr)$ for a generic point $(x,y)\in\co\times\co$.
In this case, $\BT_x(\co)=[\g,x]$ and
\[ 
     \bigl(\BT_x(\co)+\BT_y(\co)\bigr)^\perp=([\g,x]+[\g,y])^\perp=\g^x\cap\g^y=\g^{(x,y)} . 
\] 
If $z=(x,y)\in\co\times\co$ is generic, then $\dim\g^{z}$ is minimal possible and hence
\\[.6ex]
\centerline{
$\dim \bigl(\BT_x(\co)+\BT_y(\co)\bigr)=\dim\g-\dim\g^z=\max_{z\in\co\times\co} \dim G{\cdot}z$.}
 
\noindent By Lemma~\ref{lm:doubled-diag}, the diagonal action  $(G:\co\times\co)$ is 
equivalent to the doubled one. Therefore $\dim\g^z=\dim S_\star$.
Using Eq.~\eqref{eq:max-dim-orb},  we then obtain
\[
  \max_{z\in\co\times\co} \dim G{\cdot}z=\dim G-\dim S_\star =2\dim\co-2c(\co)-r(\co).  
  \qedhere
\]
\end{proof}

\begin{cl}    \label{cor:sec=g}
We have $\CS(\co)=\g$ if and only if $r(\co)=\rk\,\g$. In particular, this holds for the
distinguished $G$-orbits.  
\end{cl}
\begin{proof}
By Theorem~\ref{thm:main-dim}, we have $\codims_\g\CS(\co)=\dim S_\star$. Hence $\CS(\co)=\g$ if 
and only if $S_\star$ is finite. Since $S_\star$ is reductive, finiteness is equivalent to that  $\rk\,S_\star=0$,
i.e., $r(\co)=\rk\,\g$.
For a distinguished orbit $\co$, one has $c(\co)=\dim\co-\dim\be$ and $r(\co)=\rk\,\g$, 
see~\cite[Corollary\,2.4]{p94}.
\end{proof}

\noindent
(The assertion on distinguished orbits appears in \cite[Sect.\,3]{omoda}.)
It follows from~\cite{omoda} and Corollary~\ref{cor:sec=g} that there are many non-distinguished orbits in 
the classical Lie algebras such that $r(\co)=\rk\,\g$. Actually, this is also true for the exceptional Lie 
algebras (see Section~\ref{sect:defective-exc}), and
most of the nilpotent orbits in the simple Lie algebras have this property. 

\begin{cl}   \label{cor:pro-defekt}
The secant variety $\Sec(\ov{\BP\co})$ is defective if and only if $r(\co)<\rk\,\g$. In this case,  the defect 
of  $\Sec(\ov{\BP\co})$ equals $\delta_{\ov{\BP\co}}=2c(\co)+r(\co)$.
\end{cl}
\begin{proof}
Since $r(X)>0$ for the quasi-affine $G$-varieties, we have $2\dim\co>\dim\CS(\co)$ for all
$\co\in\N/G$. Then the assertion readily follows from Eq.~\eqref{eq:aff-defect} and Corollary~\ref{cor:sec=g}.
\end{proof}

\begin{rmk}   \label{rem:Fulton-Hansen}
By Theorem~\ref{thm:main-dim}, one has $\dim\CS(\co)<2\dim\co$, i.e., 
$\dim\Sec(\ov{\BP\co})< 2\dim\BP\co+1$. Together with the renowned Fulton--Hansen theorem 
(see e.g.~\cite[Theorem\,10.1.14]{SR}), this implies that $\Tan(\ov{\BP\co})=\Sec(\ov{\BP\co})$ for all 
$\co$. Using our methods, we also obtain this equality in Theorem~\ref{thm:main-closure} below.
\end{rmk}
 
\begin{ex}    \label{ex:max-sph-full}
Every simple Lie algebra $\g$ has a unique maximal spherical nilpotent orbit 
$\osp$~\cite[Sect.\,6.1]{aif99}. Let $\eus W=\eus W(\g)$ denote the {\it Weyl group\/} of $\g$. If $\eus W$ 
contains $-1$, then $\dim\osp=\dim\be$ and $r(\osp)=\rk\,\g$. Then $\CS(\osp)=\g$. As is well known, 
$-1\in \eus W$ if and only if $\g\not\in \eus A_n \ (n{\ge} 2)$, $\eus D_{2n+1} \ (n{\ge} 2)$, $\eus E_6$. 
\\ \indent
\textbullet \ For the exceptional algebras $\eus E_7, \eus E_8, \eus F_4$, and $\eus G_2$, the orbit
$\osp$ equals $4\mathsf A_1, 4\mathsf A_1, \mathsf A_1+\tilde{\mathsf A}_1$, and $\tilde{\mathsf A}_1$, 
respectively. 
\\ \indent
\textbullet \ 
For the classical series with $-1\in \eus W(\g)$, the partitions corresponding to $\osp$ are
\begin{itemize}
\item[--] \ $(3,2^{2m})$ \ for $\eus B_{2m+1}$ $(m\ge 1)$;
\item[--] \ $(3,2^{2m-2},1,1)$ \ for $\eus B_{2m}$ $(m\ge 2)$;
\item[--] \ $(2^{m})$ \ for $\eus C_{m}$ $(m\ge 1)$;
\item[--] \ $(3,2^{2m-2},1)$ \ for $\eus D_{2m}$ $(m\ge 2)$.
\end{itemize}
\end{ex}

\begin{ex}    \label{ex:max-sph-less} If $-1\not\in \eus W(\g)$, then the following holds:
\begin{itemize}
\item for $\eus A_{2m-1}$, the partition of $\co_{\sf sph}$ is $(2^m)$ and $r(\co_{\sf sph})=m$;
\item for $\eus A_{2m}$, the partition of $\co_{\sf sph}$ is $(2^m,1)$ and $r(\co_{\sf sph})=m$;
\item for $\eus D_{2m+1}$, the partition of $\co_{\sf sph}$ is $(3,2^{2m-2},1^3)$ and $r(\co_{\sf sph})=2m$;
\item for $\eus E_6$, one has $\co_{\sf sph}=3\mathsf A_1$ and $r(\co_{\sf sph})=4$.
\end{itemize}
In these cases, $\osp$ is defective.
\end{ex}

\begin{prop}   \label{prop:r=r'}
Suppose that\/ $\co',\co\in\N/G$, \ $\co'\subset\bco$, and  $r(\co')=r(\co)$. Then 
\beq    \label{eq:sovpad}
         \dim\co-\dim\co'=c(\co)-c(\co')
\eeq 
and $\CS(\co)=\CS(\co')$.
\end{prop}
\begin{proof} 
Since $\co$ is quasi-affine, a generic stabiliser $\be_\star=\be_\star(\co)$ for the action $(B:\co)$ is 
fully determined by the subspace in $\te^*$ generated by the weight monoid $\Gamma_\co$, 
see~\eqref{eq:t-ast} and \eqref{eq:u-ast}.

If $\co'\subset\bco$, then $\lg\Gamma_{\co'}\rg \subset \lg\Gamma_\co\rg$. Furthermore, if 
$r(\co')=r(\co)$, then
$\lg \Gamma_{\co'} \rg=\lg\Gamma_\co\rg$. This implies that $\be_\star(\co)=\be_\star(\co')$ and hence
Eq.~\eqref{eq:sovpad} holds. Using Theorem~\ref{thm:main-dim}, we conclude that 
$\dim\CS(\co)=\dim\CS(\co')$ and therefore these varieties coincide.
\end{proof}

Let $\mathsf T(\co)$ be the union of all embedded tangent  spaces to $\co$, i.e,
$\mathsf T(\co)=\bigcup_{x\in\co} \BT_{x}(\co)\subset\g$. Then 
$\mathsf T(\co)=G{\cdot}[\g,x]\subset\CS(\co)$ for some (any) $x\in\co$ and $\mathsf T(\co)\subset
\widehat{\Tan(\co)} \subset \CS(\co)$.

\begin{thm}    \label{thm:main-closure}
For any nilpotent orbit $\co$, there is a subalgebra $\te_\co\subset\te$ such that $\dim\te_\co=r(\co)$ and
\[
        \CS(\co)=\ov{\mathsf T(\co)}=\ov{G{\cdot}\te_\co} .
\]
\end{thm}
\begin{proof}
Our plan is as follows. First, we construct a certain $\te_\co$ and prove that 
\[
  \dim G{\cdot}\te_\co=2\dim\co-2c(\co)-r(\co). 
\] 
Then we show that  $\CS(\co)\supset\ov{\mathsf T(\co)}\supset\ov{G{\cdot}\te_\co}$,
which is sufficient in view of Theorem~\ref{thm:main-dim}.

1) Let $z:=(x,\bvth(x))\in \co\times\co$ be a generic point for the diagonal action of $G$ on
$\co\times\co$ that satisfy properties $(\eus P_1)$--$(\eus P_3)$. Hence $S_\star=G^z=G^x\cap G^{\bvth(x)}$, and $B^x=G^x\cap B$, etc. In terms of Lie 
algebras, we have 
\begin{itemize}
\item \ $\ess=\g^x\cap \g^{\bvth(x)}$ is a $\bvth$-stable reductive subalgebra of $\g$
and $[\te,\ess]\subset\ess$;
\item \ $\ess\cap\be=\be^x=:\be_\star$ is a Borel subalgebra of $\ess$;
\item \ $\ess\cap\ut=\ut^x=:\ut_\star$ is the nilradical of $\be_\star$;
\item \ $\ess\cap\te=\te^x=:\te_\star$ is a Cartan subalgebra of $\ess$; 
\item \ $[\g^t,\g^t]\subset\ess\subset\g^t$ for some dominant $t\in\te$.
\end{itemize} 
Recall that $\be+\es_\star=\p$ is a parabolic subalgebra and $\g^x\cap\p=\es_\star$. Then $\el_\star:=\g^t$ is the standard 
Levi subalgebra of $\p$, and the nilradical of $\p$, denoted $\n$, has the property that
$\ut=\ut_\star\oplus \n$. Let $\te_\co$ be the orthogonal complement of $\te_\star$ in $\te$. Then 
$\be_\co=\te_\co\oplus \n$ is a solvable Lie algebra, $\el_\star=\es_\star\oplus\te_\co$, and 
$\be=\be_\star\oplus \be_\co$.
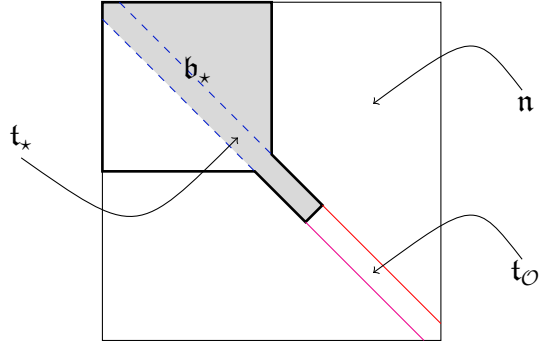
\begin{center}
\begin{tikzpicture}[scale= .45]
\draw (0,0)  rectangle (10,10);
\path[fill=gray!30] (0,9.5) -- (6,3.5) -- (6.5,4) -- (5,5.5) -- (5,10) -- (0,10)--cycle ;
\path[draw,line width=1pt] (0,5) -- (4.5,5) -- (6,3.5) -- (6.5,4) -- (5,5.5) -- (5,10) -- (0,10)--cycle ;
\draw[magenta]  (6,3.5) -- (9.5,0) ;
\draw[red]  (6.5,4) -- (10,0.5) ;
\draw[dashed,darkblue]  (0,9.5) -- (4.5,5) ;
\draw[dashed,darkblue]  (0.5,10) -- (5,5.5) ;
\draw (-2.4,6)  node {$\te_\star$} ;
\draw[->]   (-2.4, 5.4) .. controls (1,3) .. (4,6);
\draw (12.5,2)  node {$\te_{\co}$} ;
\draw[->]   (12.4, 2.4) .. controls (11,4.2) .. (8,2);
\draw (12.5,7)  node {$\n$} ;
\draw[->]   (12.4, 7.4) .. controls (11,9.2) .. (8,7);
\draw (2.8,8)  node {$\be_\star$} ;

\end{tikzpicture}
\end{center}
\vskip1.5ex
A schematic position of these subalgebras is depicted above for $\g=\slN$. Here $\be$ (resp. $\te$) is 
the set of upper-triangular (resp. diagonal) matrices. The area with the thick boundary represents 
$\es_\star$ and the shaded area is $\be_\star$. The sum $\be_\co=\te_\co\oplus\n$ is a Levi 
decomposition for $\be_\co$ and $\dim\te_\co=\dim\te-\rk\,\es_\star=r(\co)$. Since 
$[\el_\star,\el_\star]\subset\es_\star\subset\el_\star$, the toral subalgebra $\te_\co$ is contained in the 
centre of $\el_\star$, denoted $\z(\el_\star)$. An element of $\z(\el_\star)$ is said to be {\it regular}, if its 
centraliser in $\g$ equals $\el_\star$. In our case, the orbit $\co$ is quasi-affine. Then $\te_\co$ always contain regular elements of $\z(\el_\star)$, 
see \cite[Lemma\,2.2.2]{p99}. (This also follows from~\eqref{eq:u-ast}.)
Hence these regular elements are dense in $\te_\co$ and
\[
    \dim \ov{G{\cdot}\te_\co}=\dim\g-\dim\el_\star+\dim\te_\co=2\dim\n+r(\co) .
\]
Because $\dim\co=\dim\be_\co+c(\co)=\dim\n+r(\co)+c(\co)$, we obtain
$\dim \ov{G{\cdot}\te_\co}=2\dim\co+2c(\co)+r(\co)$.

2) Since $x\in\co$ is $B$-generic, one has $\codim_\g(\g^x+\be)=c(\co)$. On the other hand,
$\g^x+\be=\g^x+\p=\g^x\oplus\be_{\co}$. Hence $[\g,x]+\be_\co^\perp=\g$ and
$\dim([\g,x]\cap\be_\co^\perp)=c(\co)$. 
Note also that $\be_\co^\perp=\es_\star\oplus\n\subset \el_\star\oplus\n=\p$, i.e., 
$\p=\be_\co^\perp\oplus \te_\co$. Therefore, we obtain
\[    
      \dim\bigl( [\g,x]\cap (\es_\star\oplus\n)\bigr)=c(\co)  \ \text{ and } \
                                     \dim\bigl( [\g,x]\cap \p\bigr)=c(\co)+r(\co). 
\]
This means that there is a vector space $\eus H\subset\p$ such that $\dim\eus H=r(\co)$,
\[ [\g,x]\cap \p=\eus H\oplus \bigl( [\g,x]\cap (\es_\star\oplus\n)\bigr),
\]  
and $\eus H$ projects bijectively onto $\te_\co$. 
Since $\eus H\subset [\g,x]=(\g^x)^\perp$ and $\es_\star\subset\g^x$, we actually have
$\eus H\subset \es_\star^\perp\cap\p=\te_\co\oplus\n$. Therefore, there is a linear map $\gamma:\te_\co\to\n$
such that 
\[
     \eus H=\{t+\gamma(t)\mid t\in \te_\co\} .
\]
If $t\in\te_\co$ is a regular element of $\z(\el_\star)$, then
$U{\cdot}t=(\exp\n){\cdot}t=t+\n$. In particular, $t$ and $t+\gamma(t)$ belong to the same $G$-orbit.
Therefore $\ov{G{\cdot} \eus H}=\ov{G{\cdot}\te_\co}$. It follows that 
\[
   \ov{G{\cdot}\te_\co}\subset \ov{G{\cdot}([\g,x]\cap\p)}\subset \ov{\mathsf T(\co)}\subset  \CS(\co).
\]
Since $\CS(\co)$ is irreducible and $\dim \ov{G{\cdot}\te_\co}=\dim \CS(\co)$, we are done.
\end{proof}

By Theorem~\ref{thm:main-closure}, we have $\CS(\co)=\g$ if and only if $\te_\co=\te$, i.e., 
$\es_\star=\{0\}$. 

\begin{ex}    \label{ex:defekt-omin}
The minimal nilpotent orbit $\omin$ is spherical and $r(\omin)=1$. If $e\in\omin$ 
and $h$ is the dominant characteristic of $e$, then $\te_{\omin}=\lg h\rg$. Here we recover the known 
result that $\CS(\omin)=\ov{G{\cdot}\lg h\rg}$ and if $\rk\,\g>1$, then 
$\delta_{\ov{\BP\omin}}=2c(\omin)+r(\omin)=1$~\cite{koy99}.
\end{ex}

Note that 
$\es_\star$ is a generic isotropy subalgebra for the $G$-action on $\co\times\co$, but not on 
$\CS(\co)$. Actually, Theorem~\ref{thm:main-closure} implies the following.

\begin{cl}  \label{cor:gis-CS}
{\sf (1)} \ A generic isotropy subalgebra for $(G:\CS(\co))$ is equal to $\el_\star=\es_\star\oplus {\te_\co}$; 
\\ \indent {\sf (2)} \ $\dim\CS(\co)\md G=r(\co)$. 
\end{cl}
\begin{proof}
{\sf (1)} \ It is already proved that $\CS(\co)=\ov{G{\cdot}\te_\co}$ and $\g^x=\el_\star$ for almost all 
$x\in\te_\co$.

{\sf (2)} \ Generic $G$-orbits in $\CS(\co)$  are closed and of dimension $\dim\g-\dim\el_\star=2\dim\n$. 
Hence
$\dim\CS(\co)\md G=\dim\CS(\co)-2\dim\n=\dim \g-\dim\es_\star-2\dim\n=\dim\te_\co$.
\end{proof}

\begin{rmk}   \label{rem:Dixmier}
If $\es_\star$ is semisimple, then $\te_\co=\z(\el_\star)$ is the centre of $\el_\star$. In this case, 
$\CS(\co)$ is the closure of the {\it Dixmier sheet\/} $\gD(\el_\star)$ corresponding to $\el_\star$. We refer 
to~\cite[\S\,5]{bokr} and~\cite{schichten} for a thorough treatment of sheets in $\g$ (cf. also~\cite[\S\,6.10]{vp}). Then 
$\CS(\co)\cap\N$ is irreducible, and the 
dense $G$-orbit in this intersection is the nilpotent orbit of $\gD(\el_\star)$. Furthermore, if $\es_\star$ is 
not semisimple, then $\te_\co$ still contains regular elements of $\z(\el_\star)$. Therefore, 
$\CS(\co)\cap \gD(\el_\star)$ is dense in $\CS(\co)$ and hence $\CS(\co)\cap\N$ is still is the closure 
of the nilpotent orbit of $\gD(\el_\star)$. Thus, the following is true:
\end{rmk}
\begin{cl}  \label{cl:Dixmier}
For any $\co\in\N/G$, the dense nilpotent orbit in $\CS(\co)\cap\N$, say $\tilde\co$, is Richardson.
The closure of $\tilde\co$ is ${G{\cdot}\n}$ and
$\dim\tilde\co=\dim\CS(\co)-r(\co)=\dim\g-\dim\el_\star$.
\end{cl}

\noindent
Recall that $\BT^*(\co)$ is the cotangent bundle of $\co$. It is a symplectic $G$-variety equipped with
a Hamiltonian $G$-action.
\begin{thm}    \label{thm:secant-mu}
Let $\mu: \BT^*(\co)\to \g^*$ be the moment map. Then under the identification $\g^*\simeq \g$, we
have  $\CS(\co)=\ov{\Ima(\mu)}$.
\end{thm}
\begin{proof}
The tangent space at $x\in\co$ is $[\g,x]\subset \g$. Here $[\g,x] \simeq \g/\g^x$ as $G^x$-module, and 
the cotangent space at $x$ is $\Ann(\g^x)=\{\xi\in\g^*\mid \xi(y)=0 \ \ \forall y\in\g^x\}$. Identifying $\g$ and
$\g^*$ via the Killing form on $\g$, we obtain $\Ann(\g^x)=(\g^x)^\perp=[\g,x]$. Upon these identifications,
we also have  $\ov{\Ima(\mu)}=\ov{G{\cdot}([\g,x])}=\ov{\mathsf T(\co)}$.
\end{proof}

\begin{rmk}
A Lie algebra $\q=\Lie Q$ is said to be {\it quadratic}, if it has a non-degenerate symmetric invariant 
bilinear form. The previous proof shows that, for a quadratic Lie algebra $\q$, the tangent and cotangent 
bundles of any $Q$-orbit in the $Q$-module $\q\simeq\q^*$ are isomorphic. Furthermore, if $\co$ is a 
conical $Q$-orbit in $\q^*$, then the closure of the image of the moment map $\mu: \BT^*(\co)\to \q^*$ is
the affine cone over the tangential variety $\Tan(\ov{\BP\co})\subset\BP\q$.
\end{rmk}

\section{The defective nilpotent orbits in the classical Lie algebras}
\label{sect:4}
\noindent
In this section, we describe results of Omoda~\cite{omoda} on $\CS(\co)$ for nilpotent orbits in the 
classical Lie algebras and provide some complements to those results if $\g=\soN$. We also point out 
the maximal defective nilpotent orbits. Recall that $\co$ is defective if and only if $r(\co)< \rk\,\g$ 
(Corollary~\ref{cor:pro-defekt}). This also means that $\es_\star\ne\{0\}$ and $\CS(\co)\ne \g$.  

Let $\g=\g(\eus V)$ be a classical Lie algebra, where $\eus V$ is the space of tautological representation 
of $\g$. For $A\in\g(\eus V)\subset\glv$, let $\rank(A)$ denote the usual rank of a matrix of order 
$\dim\eus V$ and $A^i$ the $i$-th matrix power of $A$.
If $\blb$ is an admissible partition of $\dim \eus V$, then
$\co_\blb$ stands for the associated nilpotent orbit in $\g(\eus V)$. If $\g=\sone$ and
$\blb$ is {\it very even}, i.e., all parts of $\blb$ are even, then there are two associated orbits, denoted  
$\co_{\blb,\mathsf{I}}$ and $\co_{\blb, \mathsf{II}}$.

\subsection{The case of $\g=\slN$}
Let $\blb=(\lb_1,\lb_2,\dots,\lb_p)$ be a partition of $N$, i.e., $\lb_1\ge\dots\ge\lb_p$ and
$\sum_i \lb_i=N$. It is assumed that $\lb_p\ge 1$, i.e., $p$ equals the number of non-trivial parts of 
$\blb$. If $A\in\co_\blb$, then $\rank(A)=N-p$ and $\lb_1$ is the minimal integer such that $A^{\lb_1}=0$.
For any $\co\in \N/SL_N$, the conical secant variety of $\co$ is described by Omoda 
in~\cite[Sect.\,4]{omoda}. His result can be summarised as follows: 

\begin{itemize}
\item[$(\lozenge_1)$] \ \emph{If $\lb_1\ge 3$ and $A\in\co_\blb$, then $\CS(\co_\blb)=\{M\in\slN\mid \rank(M)\le 2\rank(A)\}$};
\item[$(\lozenge_2)$] \ \emph{Let $\chi_M(t)=\det(tI-M)=t^N-\sigma_1(M)t^{N-1}+\dots +(-1)^N\sigma_N(M)$ be
the characteristic polynomial of $M\in\slN$. If $\blb=(2^i, 1^{N-2i})$, 
then \\ $\CS(\co_\blb)=\{M\in\slN \mid \rank(M)\le 2\rank(A)=2i \ \& \
\sigma_{2j-1}(M)=0 \text{ for }2\le j\le i\}$.} 
\end{itemize}

\noindent
Note that $\lb_1\ge 3$ if and only if $A^2\ne 0$. It is known that $\co_\blb$ is spherical if and only if 
$\lb_1=2$, and then $r(\co_\blb)=\rank(A)$~\cite[Sect.\,4]{p94}. 
Hence a peculiarity in $(\lozenge_2)$ is related exactly to the spherical nilpotent orbits of $SL_N$.
Using $(\lozenge_1)$ and $(\lozenge_2)$, one easily derives  the following assertion.

\begin{prop}   \label{prop:sln-max}
For $\g=\slN$, the maximal nilpotent orbits such that $\CS(\co)\ne \slN$ correspond to the partitions:
\begin{itemize}
\item $(2,2)$, if $N=4$;
\item $(2^n)$ and $(n,1^n)$, if $N=2n\ge 6$;
\item $(n+1, 1^n)$, if $N=2n+1\ge 5$.
\end{itemize}
\end{prop}

\subsection{The case of $\g=\spn$}
The nilpotent $Sp_{2n}$-orbits are parameterised by the partitions of $2n$ such that each odd part 
occurs an even number of times~\cite[Ch.\,5.1]{CM}. 
For any $\co\in \N/Sp_{2n}$, the conical secant variety of $\co$ is described in~\cite[Sect.\,6]{omoda}. 
Omoda's result can be stated as follows:
\\[.6ex]
\hbox to \textwidth{  $(\lozenge_3)$ \hfil
\emph{if $A\in\co_\blb\subset \N(\spn)$, then $\CS(\co_\blb)=\{M\in\spn\mid \rank(M)\le 2\rank(A)\}$.}
\hfil}
\vskip.6ex

\noindent
If $\blb=(\lb_1,\dots,\lb_p)$ is a ``symplectic'' partition and $A\in\co_\blb$, then $\rank(A)=2n-p$. 
Using this fact and $(\lozenge_3)$, one readily obtains the following assertion.

\begin{prop}   \label{prop:spn-max}
For $\g=\spn$, the maximal nilpotent orbits such that $\CS(\co)\ne \spn$ correspond to the partitions:
\begin{itemize}
\item $(2m, 1^{2m})$, if $n=2m$;
\item $(2m, 2, 1^{2m})$, if $n=2m+1$.
\end{itemize}
\end{prop}
A nilpotent $Sp_{2n}$-orbit $\co_\blb$ is spherical if and only if $\lb_1=2$, i.e., $A^2=0$~\cite[Sect.\,4]{p94}.

\subsection{The case of $\g=\soN$}
The nilpotent $O_N$-orbits are parameterised by the partitions of $N$ such that each even part occurs 
an even number of times~\cite[Ch.\,5.1]{CM}. The $O_N$-orbits coincide with $SO_N$-orbits unless the partition $\blb$ is {\it very even}. If $\blb$ is very even,
then the $O_N$-orbit splits into two associated $SO_N$-orbits, denoted $\co_{\blb,{\sf I}}$ and $\co_{\blb, {\sf II}}$.

For $\soN$, results of Omoda are incomplete. By~\cite[Thm.\,5.4]{omoda}, one has 
\\[.6ex]
\hbox to \textwidth{  $(\lozenge_4)$ \hfil
\emph{If $A\in\co_\blb\subset\N(\soN)$ and $\lb_1\ge 3$, then $\CS(\co_\blb)=\{M\in\soN\mid \rank(M)\le 2\rank(A)\}$}.
\hfil}
\vskip.6ex

\noindent
However, if $\lb_1=2$, then no description of $\CS(\co_{\blb})$ is given therein. In particular, even
$\dim\CS(\co_\blb)$ is not determined. We say more about this case in Example~\ref{ex:ex}.

The nilpotent $SO_N$-orbits with  $\lb_1=2$ are spherical, but unlike the case of $SL_N$ or $Sp_{2n}$,
these are not all spherical orbits. (The orbit $\co_\blb\in \N/SO_N$ is spherical if and only if 
$\lb_1+\lb_2\le 5$.) If $\blb=(2^{2m},1^l)$, then $r(\co_{\blb})=m$, see~\cite[(4.4)]{p94}. Here $N=4m+l$ 
and $\rk\,\g=2m+[l/2]$. (Note that there are two such orbits, if $l=0$.) By Corollary~\ref{cor:sec=g}, this 
means that all such $SO_N$-orbits are defective. 

If $\blb=(\lb_1,\dots,\lb_p)$ is an ``orthogonal'' partition and $A\in\co_\blb$, then $\rank(A)=N-p$ is 
always even. Combining this with $(\lozenge_4)$ yields the following assertion

\begin{prop}   \label{prop:son-max}
For $\g=\soN$, the maximal nilpotent orbits such that $\CS(\co)\ne \soN$ correspond to the partitions:
\begin{itemize}
\item $(2m-1, 1^{2m+1})$ and $(2^{2m})$, \ if\/ $\g=\mathfrak{so}_{4m}$;
\item $(2m-1, 1^{2m+2})$ and $(2^{2m},1)$, \ if\/ $\g=\mathfrak{so}_{4m+1}$;
\item $(2m+1, 1^{2m+1})$, \quad if\/ $\g=\mathfrak{so}_{4m+2}$;
\item $(2m+1, 1^{2m+2})$, \quad if\/ $\g=\mathfrak{so}_{4m+3}$.
\end{itemize}
\end{prop}

\noindent
Since the partition $(2^{2m})$ is very even, one has three maximal defective orbits in $\eus D_{2m}$.

\subsection{Complements}   \label{subs:ex}
Results of~\cite{omoda} exposed in $(\lozenge_1)$-$(\lozenge_4)$ above show that, for most nilpotent 
orbits $\co\subset \g(\eus V)$, the cone $\CS(\co)$ is a determinantal variety. These varieties are 
explicitly used in Section~\ref{sect:comput-c-r}, and
some standard results on them are gathered in Appendix~\ref{appendix}.

\begin{ex}    \label{ex:ex}
Using our methods, we provide more information on $\CS(\co_\blb)$ for the case left open 
in~\cite{omoda}, i.e., $\g=\mathfrak{so}_{4m+l}$ and $\blb=(2^{2m},1^l)$.
Here $\dim\co_\blb=4m^2+2ml-2m$ and $r(\co_\blb)=m$. By Theorem~\ref{thm:main-dim},
\[ 
    \dim\CS(\co_\blb)=2\dim\co_\blb-m=8m^2+4ml-5m , 
\] 
hence $\codim_{\soN}\CS(\co_\blb)=3m+\genfrac{(}{)}{0pt}{}{l}{2}$. Since $\rank(A)=2m$ for 
$A\in\co_\blb$ and $\soN$ can be regarded as the space of skew-symmetric matrices of order $N$, $\CS(\co_\blb)$ is 
contained in the variety of skew-symmetric matrices $M$ such that $\rank(M)\le 4m$.

For the skew-symmetric determinantal variety $\sk_{4m}(N)$, we have
\[
    \dim \sk_{4m}(N)=2m(2N-4m-1), 
\]
see Appendix~\ref{appendix}.
Hence $\codim_{\soN} \sk_{4m}(N)=\genfrac{(}{)}{0pt}{}{N-4m}{2}$. This shows that the 
codimension of $\CS(\co_\blb)$ in $\sk_{4m}(N)$ equals $3m$. It is not clear, though, what are the 
equations of $\CS(\co_\blb)$ inside $\sk_{4m}(N)$. Nevertheless, one can explicitly describe 
the subalgebra $\te_\co\subset\te$.

The weight monoid  $\Gamma_{\co_\blb}$ is freely generated by the dominant weights 
$\mu_1,\dots,\mu_m$, where $\mu_i=\varpi_{2i}$ \ if $i<m$ and
$\mu_m=\begin{cases} 2\varpi_{2m}, &\text{ if } \ l\le 1,\\
\varpi_{2m}+\varpi_{2m+1}, &\text{ if } \ l=2 ,\\
\varpi_{2m}, &\text{ if } \ l\ge 3 .   \end{cases}$

\noindent
(See \cite[Sect.\,6]{p03}.) More accurately, if $l=0$, i.e., $\g=\eus D_{2m}$, then there are two associated  
``very even'' orbits. For $\co_{(2^{2m}), {\sf I}}$, one has $\mu_m=2\varpi_{2m}$
(as indicated above), while $\mu_m=2\varpi_{2m-1}$ for $\co_{(2^{2m}), {\sf II}}$. 
For the model of $\soN$ as the set of skew-symmetric matrices with respect to the
anti-diagonal, the Cartan subalgebra $\te$ consists of the diagonal matrices in $\soN$. That is, 
\[
  \te=\{\mathsf{diag}(\esi_1,\dots,\esi_n, (0), -\esi_n,\dots,-\esi_1)\},
\]
where $n=[N/2]$ and $0$ occurs only if $N$ is odd, i.e., $\g$ is of type $\eus B_n$. Using this model and 
the standard formulae for $\{\varpi_j\}$ via $\{\esi_i\}$ in the orthogonal case (\cite[Table\,1]{t41}), one readily 
obtains $\te_{\co_\blb}$ as the linear span of $\mu_1,\dots,\mu_m$.  Here
\[
  \te_{\co_\blb}=\{\mathsf{diag}(\nu_1,\nu_1,\dots,\nu_m,\nu_m,\underbrace{0,\dots,0}_{l}, -\nu_m,-\nu_m,\dots,-\nu_1,-\nu_1)\}.
\]
If $l=0$, then this formula is valid for $\co_{(2^{2m}), {\sf I}}$, while for $\co_{(2^{2m}), {\sf II}}$ the 
central part in the formula should be $\dots,\nu_m,-\nu_m, \nu_m,-\nu_m,\dots$. For $\co_\blb$, the 
generic isotropy subalgebra $\es_\star$ for the diagonal action $(G:\co_\blb\times\co_\blb)$ is equal to 
$(\tri)^m\oplus \mathfrak{so}_l$, see proof of Theorem~\ref{thm:so}{\sf (iii)}. Hence $\es_\star$ is 
semisimple unless $l=2$. Therefore, if $l\ne 2$, then $\CS(\co_\blb)$ is the closure of a Dixmier sheet,  
see Remark~\ref{rem:Dixmier}. For $l=2$, the centre of $\es_\star$ is 1-dimensional. Hence 
$\CS(\co_\blb)$ is of codimension~1 in the Dixmier sheet corresponding to the Levi subalgebra 
$\el_\star=\es_\star\oplus\te_\co$ whose semisimple part is $(\tri)^m$.

By $(\eus P_2)$-$(\eus P_3)$ in Section~\ref{subs:twist}, for the canonical embedding 
$\ess\hookrightarrow\g$, the set of simple roots of $\ess$ is a subset of $\Pi$. Moreover, by
Eq.~\eqref{eq:t-ast} and \eqref{eq:u-ast}, the knowledge of $\Gamma_{\co_\blb}$ yields 
the canonical  embedding of $\es_\star$. For $\ess= (\tri)^m\oplus \mathfrak{so}_l$, we get the following. 
If $l\ge 1$ and $l\ne 2$, then the simple roots of $\ess$ are 
\beq   \label{eq:roots}
\underbrace{\ap_1,\ap_3,\dots,\ap_{2m-1}}_{(\tri)^m},
\underbrace{\ap_{2m+1},\dots,\ap_n}_{\mathfrak{so}_l},
\eeq
where $n=2m+[l/2]$ and the natural numbering of simple roots is used, cf.~\cite[Table\,1]{t41}. If $l=0$, then there are two very even orbits and the canonical embedding of
$(\tri)^m$ is given by the roots 
$\begin{cases}  \ap_1,\ap_3,\dots, \ap_{2m-3},\ap_{2m-1} & \text{ for } \co_{(2^{2m}), {\sf II}}, \\
 \ap_1,\ap_3,\dots,\ap_{2m-3},\ap_{2m} & \text{ for } \co_{(2^{2m}), {\sf I}} \ .
\end{cases}$ \\
Finally, if $l=2$, i.e., $\g=\mathfrak{so}_{4m+2}$, then then the roots of $(\tri)^m$ are as in
\eqref{eq:roots}
and the toral algebra $\mathfrak{so}_2$ is
$\{\mathsf{diag}(0,\dots,0,\esi_{2m+1},-\esi_{2m+1},0,\dots,0)\}$. 
\end{ex}

\section{Computing the complexity and rank}
\label{sect:comput-c-r}

\noindent
In this section, we determine the generic isotropy subalgebra $\es_\star$, and thereby the rank and 
complexity for {\bf all} nilpotent orbits $\co$ in the {\bf classical} Lie algebras. This allows us to describe
$\CS(\co)$ independently of~\cite{omoda}.

Set $\N_o=\N\setminus\{0\}$. Our main output is that the poset ${\N}_o/G$ can be split into a 
disjoint union of subposets $\{\ups_\xi\}_{\xi\in\Xi}$ such that each $\ups_\xi$ has a unique maximal element
and the mapping $\co\mapsto r(\co)$ is constant on $\ups_\xi$.
More precisely, we show that the generic isotropy subalgebra $\es_\star$ for the action
$(G:\co\times\co)$ is the same for all $\co\in\ups_\xi$. This implies that the map $\co\mapsto r(\co)$ is 
constant on $\ups_\xi$. Then Proposition~\ref{prop:r=r'} shows that all $\co\in\ups_\xi$ have the 
same variety $\CS(\co)$. The information on $\es_\star$ and its embedding into $\g$ allows us to 
determine, in principle, the subalgebra $\te_\co$. 

The definition of the posets $\{\ups_\xi\}$ is inspired by results of~\cite{omoda}, which show that, for 
most of nilpotent orbits $\co$ in the classical algebras, the variety $\CS(\co)$ depends only 
on $\rank(A)$ for $A\in\co$ (see exposition in Section~\ref{sect:4}). This provides a clue for a splitting of 
$\N_o/G$ into the posets $\{\ups_\xi\}_{\xi\in\Xi}$. Once the splitting is described, the subsequent proofs 
are completely independent of \cite{omoda}. They exploit the $(\BZ,h)$-grading of $\g$ associated 
with $\co$ and Theorem~\ref{thm:p94}. Therefore, our results provide an independent confirmation and 
strengthening of results of Omoda. The explicit description of posets $\{\ups_\xi\}_{\xi\in\Xi}$ is given via 
partitions.
\\ \indent
We prove below three theorems for the series of classical Lie algebras. The scheme and ideas of proof 
are the same for all of them. Therefore, the maximal amount of details is given in the first proof (for 
Theorem~\ref{thm:sl} with $\g=\slN$).

\subsection{More notation} To discuss generic isotropy subalgebras for representations related to 
nilpotent orbits and their $(\BZ, h)$-gradings, an additional notation is required. Recall that 
$\{\varpi_i\}$ are the fundamental weights of $\g$.

\textbullet \ \ The fundamental weights of the simple factors of $\el=\g(0)$ are denoted by $\{\vp_i\}$
for the first factor, $\{\vp'_j\}$ for the second factor, etc. Write $\te_m$ for an $m$-dimenonal 
diagonalisable subalgebra (e.g. the centre of $\el$). Then $\{\esi_1,\dots,\esi_m\}$ is  a suitable basis for $\te^*_m$.

\textbullet \ \ The simple $\el$-modules are identified with their highest weights, via the multiplicative 
notation. For instance, any monomial in $\{\esi_i\}$ represents a 1-dimensional $\te_m$-module, and
the simple $\el$-module with highest weight $2\vp_i+\vp'_j+\esi_k+\esi_l$ is denoted by 
$\vp_i^2\vp'_j\esi_k\esi_l$. Finally, $\odin$ stands for the trivial 1-dimensional representation.

\textbullet \ \ To describe some varieties $\CS(\co)$, we need the determinantal varieties
$\sD_r(N)$, $\sD_r^o(N)$, $\sk_{2m}(N)$, and $\sym_m(N)$, see Appendix~\ref{appendix}.

\subsection{$\g=\slN$}   \label{subs:sl}
Consider the following subposets of $\N/SL_N$.
\begin{enumerate}
\item $\ups_{\sf max}=\{\co\mid \rank(e)\ge N/2 \ \ \& \ \text{ $e^2\ne 0$ for $e\in\co$}\}$;
\item $\ups_j=\{\co\mid \rank(e)=j < N/2 \ \ \& \ \text{ $e^2\ne 0$ for $e\in\co$}\}$; 
\item $\ups_{(1,r)}=\{\co_{(2^r, 1^{N-2r})}\}$ for $r=1,2,\dots, [N/2]$.
\end{enumerate}

\noindent
Hence $2\le j\le [\frac{N-1}{2}]$ in (2), \ each $\ups_{(1,r)}$ consists of a single spherical $SL_N$-orbit, and 
\[
    {\N}_o/SL_N=\ups_{\sf max}\sqcup (\bigsqcup_{j=2}^{[(N{-}1)/2]}\ups_j)\sqcup (\bigsqcup_{r=1}^{[N/2]}\ups_{(1,r)}) .
\]
Let $\blb=(\lb_1,\dots,\lb_p)$ be a partition of $N$. Then
\begin{itemize}
\item \ $\co_\blb\in\ups_{\sf max}$ if and only if $N-p\ge N/2$ and $\lb_1\ge 3$. Hence $N\ge 3$ and the 
minimal element of $\ups_{\sf max}$ corresponds to the partitions $(3, 2^{n})$, if $N=2n+3$\ \& \ $(3, 2^{n},1)$, if $N=2n+4$. Clearly, $\co_{\sf reg}$ is the only maximal element of $\ups_{\sf max}$.
\item \ $\co_\blb\in\ups_{j}$ if and only if $N-p=j<N/2$ and $\lb_1\ge 3$. Note that subsets $\ups_j$ occur 
only for $N\ge 5$. Each $\ups_j$ has a unique minimal and unique maximal element, their 
partitions being $\blb_{\sf min}(j)=(3, 2^{j-2}, 1^{N+1-2j})$ and  $\blb_{\sf max}(j)=(j+1, 1^{N-j-1})$. Write
$\co_{\sf min}(j)$ and $\co_{\sf max}(j)$ for the corresponding orbits.
\end{itemize} 

\begin{thm}  \label{thm:sl}
Let $e\in\N(\slN)$ and $\co=SL_N{\cdot}e$.
\begin{itemize}
\item[\sf (i)] \ If\/ $\co\in\ups_{\sf max}$, then $\es_\star=\{0\}$, $r(\co)=\rk\,\slN=N-1$, and 
$\CS(\co)=\slN$;
\item[\sf (ii)] \ If\/ $\co\in\ups_j$ $(2\le j\le [\frac{N-1}{2}])$, then $\es_\star=\mathfrak{gl}_{N-2j}$, $r(\co)=2j-1$, and 
$\CS(\co)=\sD_{2j}^o(N)$;
\item[\sf (iii)] \ If\/ $\co=\co_{(2^r, 1^{N-2r})}$, then $\co$ is spherical and\/ $1\le r\le N/2$. Here
$\es_\star=\mathfrak{gl}_{N-2r}\oplus\te_{r-1}$, \\ 
$r(\co)=\rank(e)=r$, and \\ 
$\CS(\co)=\sD_{2r}^o(N)\cap \{\sigma_3=\dots =\sigma_{2r-1}=0\}=\sD_{2r}(N)\cap \{\sigma_1=\sigma_3=\dots =\sigma_{2r-1}=0\}$.
\end{itemize}
\end{thm}
\begin{proof}
{\sf (i)} It suffices to prove that $\es_\star=\{0\}$ for the minimal element of $\ups_{\sf max}$.

$(\blacklozenge_1)$ For $N=2n+3$, we have $\blb=(3,2^n)$. If $n=0$, then $\co_\blb=\co_{\sf reg}$ 
is distinguished. Hence one may assume that $n\ge 1$. The dominant characteristic of $e\in\co_\blb$ 
is the diagonal matrix
\[
   h=\mathsf{diag}(2,\underbrace{1,\dots,1}_n,0,\underbrace{-1,\dots,-1}_n,-2)
\]
and hence the {\it weighted Dynkin diagrams} (={\it\bfseries wDd}) of $e$ is
$\eus D(\co_\blb)=(1\,\underbrace{0 \dots 0}_{n-1}\,1\,1\,\underbrace{0 \dots 0}_{n-1}\,1)$. The 
associated $(\BZ, h)$-grading is readily determined via $h$ or $\eus D(\co_\blb)$, see 
Figure~\ref{fig:Z-grad}, where the nonzero entries of $h$ are placed on the diagonal.
\begin{center}
\begin{figure}[ht]   
\begin{tikzpicture}[scale= .45]
\draw (0,0)  rectangle (11,11);
\path[draw,dashed]  (0,1) -- (11,1); 
\path[draw,dashed]  (0,5) -- (11,5); 
\path[draw,dashed]  (0,6) -- (11,6); 
\path[draw,dashed]  (0,10) -- (11,10); 
\path[draw,dashed]  (1,0) -- (1,11); 
\path[draw,dashed]  (5,0) -- (5,11); 
\path[draw,dashed]  (6,0) -- (6,11); 
\path[draw,dashed]  (10,0) -- (10,11); 
\path[fill=gray!30] (0,11) -- (0,10) -- (1,10) -- (1,11)--cycle ;
\path[fill=gray!30] (1,10) -- (1,6) -- (5,6) -- (5,10)--cycle ;
\path[fill=gray!30] (5,6) -- (5,5) -- (6,5) -- (6,6)--cycle ;
\path[fill=gray!30] (6,5) -- (6,1) -- (10,1) -- (10,5)--cycle ;
\path[fill=gray!30] (10,1) -- (10,0) -- (11,0) -- (11,1)--cycle ;
\draw (3.5,8.5)  node {$\sln$} ;  \draw (8.6,3.3)  node {$\sln$} ;
\draw (3.2,10.5)  node {{\color{red}$1$}} ; \draw (8.2,5.5)  node {{\color{red}$1$}} ; 
\draw (5.45,8)  node {{\color{red}$1$}} ; \draw (10.45,3)  node {{\color{red}$1$}} ; 
\draw (5.45,10.5)  node {{\color{darkblue}$2$}} ; \draw (10.45,5.5)  node {{\color{darkblue}$2$}} ; 
\draw (8.2,8)  node {{\color{darkblue}$2$}} ; \draw (10.45,10.45)  node {$4$} ; 
\draw (8.2,10.45)  node {{\color{forest}$3$}} ; \draw (10.45,8)  node {{\color{forest}$3$}} ;
\draw (11.6,.5)  node {\footnotesize {\color{darkblue}$1$}} ;
\draw (11.6,3)  node {\footnotesize {\color{darkblue}$n$}} ;
\draw (11.6,5.5)  node {\footnotesize {\color{darkblue}$1$}} ;
\draw (11.6,8)  node {\footnotesize {\color{darkblue}$n$}} ;
\draw (11.6,10.5)  node {\footnotesize {\color{darkblue}$1$}} ;
\draw (.5,10.5)  node {{\tiny $2$}} ; 
\draw (1.5,9.5)  node {{\tiny $1$}} ; 
\draw (4.6,6.5)  node {{\tiny $1$}} ;
\draw (5.5,5.5)  node {{\tiny $0$}} ;
\draw (6.5,4.5)  node {{\tiny -$1$}} ;
\draw (9.6,1.4)  node {{\tiny -$1$}} ;
\draw (10.5, .5)  node {{\tiny -$2$}} ;
\path[draw,dashed]  (1.7,9.3) -- (4.3,6.7); 
\path[draw,dashed]  (6.8,4.2) -- (9.3,1.7); 
\end{tikzpicture}
\caption{}
\label{fig:Z-grad}
\end{figure}
\end{center}
\noindent 
Here $\el=\g(0)\simeq\sln\oplus\sln\oplus\te_4$ is represented by the shaded area, and the centre of $\g(0)$ is
\[
    \te_4= \{\mathsf{diag}(\mu_1,\underbrace{\mu_2,\dots,\mu_2}_n,\mu_3,\underbrace{\mu_4,\dots,\mu_4}_n,\mu_5)\}\cap\slN .
\]
The areas marked with `$i$' represent the subspace $\g(i)$. Set $\esi_i=\mu_i-\mu_{i+1}$ for $1\le i\le 4$. Then $\te_4^\ast=\lg \esi_1,\dots,\esi_4\rg$ and the structure of the $\g(0)$-modules $\g(i)$
is easily determined from Figure~\ref{fig:Z-grad} or $\eus D(\co_\blb)$. Namely, 
\begin{itemize}
\item \quad $\g(1)=\esi_1\vp_1+\esi_2\vp_{n-1}+\esi_3\vp'_1+\esi_4\vp'_{n-1}$;
\item \quad $\g(2)=\esi_1\esi_2+\esi_2\esi_3\vp_{n-1}\vp'_1+\esi_3\esi_4$;
\item \quad $\g(3)=\esi_1\esi_2\esi_3\vp'_1+\esi_2\esi_3\esi_4\vp_{n-1}$;
\item \quad $\g(4)=\esi_1\esi_2\esi_3\esi_4$.
\end{itemize}
The structure of $\g(1)$ is not required for subsequent computations, cf.~Theorem~\ref{thm:p94}. It is 
given here for the sake of completeness. Recall that $\ka$ is a generic isotropy subalgebra for the 
$\el$-module $\g(2)$ and $\ka$ is determined by $\blb$. Here $\ka\simeq \sln\oplus\te_1$, where 
$\sln=\Delta(\sln)$ is diagonally embedded in $\sln\oplus\sln$ and 
$\te_1=\{\esi_1=-\esi_2=\esi_3=-\esi_4\}\subset\te_4$. Therefore, $\g(2)\vert_\ka=\ad(\sln)+3\odin$, and
a generic isotropy subalgebra for $(\ka:\g(2))$ is $\es\simeq \te(\sln)\oplus \te_1$. Finally, $\es_\star$ is
a {\sf g.i.s.} for the $\es$-module $\g({\ge}3)\oplus \g({\ge} 3)^*$ and the formulae for $\g(3)$ and $\g(4)$
show that $\es_\star=\{0\}$.

$(\blacklozenge_2)$ For $N=2n+4$, we have $\blb=(3,2^n,1)$. This case can be included in the family 
of partitions considered in part {\sf (ii)} below.

{\sf (ii)} It suffices to check that $\es_\star$ is equal to $\mathfrak{gl}_{N-2j}$ for both
$\co_{\sf min}(j)$ and $\co_{\sf max}(j)$. 

$(\blacklozenge_3)$  For $\co_{\sf min}(j)$, we set $n=j-2$ and work with 
$\blb=\blb_{\sf min}(n+2)=(3,2^n,1^m)$, where $n\ge 0$, $m\ge 2$, and $N=2n+m+3$. Then 
$N-2j=m-1$. The case $n=0$ has already been considered in~\cite[Thm.\,4.5.19]{p99}. If $n\ge 1$, then
\begin{gather*}
h=\mathsf{diag}(2,\underbrace{1,\dots,1}_n,\underbrace{0,\dots,0}_{m+1},\underbrace{-1,\dots,-1}_n,-2)  
\\   \text{and} \ \ \eus D(\co_\blb)=(1\,\underbrace{0 \cdots 0}_{n-1}\,1\,\underbrace{0\cdots 0}_m\,1\,\underbrace{0 \cdots 0}_{n-1}\,1) .
\end{gather*}
The block structure of $\slN$ is similar to that in Figure~\ref{fig:Z-grad}. The only difference is that
now the central diagonal block is of size $m+1$ (instead of $1$). 
Then $\g(0)=\sln\oplus\mathfrak{sl}_{m+1}\oplus\sln\oplus\te_4$ and the $\g(0)$-modules $\g(i)$ for
$i=2,3,4$ are
\begin{itemize}
\item \quad $\g(2)=\esi_1\esi_2\vp'_1+\esi_2\esi_3\vp_{n-1}\vp''_1+\esi_3\esi_4\vp'_m$.
\item \quad $\g(3)=\esi_1\esi_2\esi_3\vp''_1+\esi_2\esi_3\esi_4\vp_{n-1}$;
\item \quad $\g(4)=\esi_1\esi_2\esi_3\esi_4$.
\end{itemize}
As always, $\g(1)$ is not needed for our purposes, and we omit it here. Note that the second simple factor of $\g(0)$,
$\mathfrak{sl}_{m+1}$,  acts trivially on $\g(3)\oplus\g(4)$. Here 
$\ka=\Delta(\sln)\oplus \te_1\oplus\mathfrak{gl}_{m}$. Then
$\es=\te(\sln)\oplus\te_1\oplus \mathfrak{gl}_{m-1}$, where the last summand lies in 
$\mathfrak{sl}_{m+1}$. Finally, we obtain $\es_\star=\mathfrak{gl}_{m-1}$, which includes the case 
with $m=1$ required in $(\blacklozenge_2)$ above.

$(\blacklozenge_4)$  For $\co_{\sf max}(j)$, the structure of the $(\BZ,h)$-grading depends on the parity 
of $j$. We only consider the case of even $j$. Write $\blb=\blb_{\sf max}(j)=(2m+1, 1^r)$, where $j=2m$, 
$N=2m+1+r$, $N-2j=r+1-2m$. Then the conditions on $m$ and $r$ are $1\le m\le [r/2]$. In this case,
\begin{gather*}
h=\mathsf{diag}(2m,\dots,4,2,\underbrace{0,\dots,0}_{r+1},-2,-4,\dots,-2m),  
\\
\eus D(\co_\blb)=(\underbrace{2 \cdots 2}_{m}\,\underbrace{0\cdots 0}_r\,\underbrace{2 \cdots 2}_{m}) ,
\end{gather*}
and the orbit $\co_\blb$ is even. We have 
$\g(0)=\mathfrak{sl}_{r+1}\oplus\te_{2m}$ and \\
\centerline{
$\g(2)=\esi_1+\dots+\esi_{m-1}+\esi_m\vp_1+\esi_{m+1}\vp_r+\esi_{m+2}+\dots+\esi_{2m}$.}

\noindent Therefore, $\ka\simeq\mathfrak{sl}_r\oplus\te_1\simeq \mathfrak{gl}_r$.
Then $\es=\mathfrak{gl}_{r-1}=\glbv$ and the $\es$-module $\g({\ge}3)\oplus\g({\ge}3)^*$ is isomorphic to
$(m{-}1)\BV+(m{-}1)\BV^*+(2m^2{+}m{-}2)\odin$. 
Hence $\es_\star=\mathfrak{gl}_{r-1-2(m-1)}=\mathfrak{gl}_{r+1-2m}$, as required.

{\bf --} \ Since $\es_\star=\mathfrak{gl}_{N-2j}$ for all $\co\in\ups_j$, we obtain $r(\co)=N-1-(N-2j)=2j-1$ 
and $\dim\CS(\co)=\dim\g-\dim\es_\star$. On the other hand, $\rank(e)=j$ for all $\co\in\ups_j$ and hence 
$\CS(\co)\subset \sD_{2j}^o(N)$. Here
$\dim\g-\dim\es_\star=\dim \sD_{2j}^o(N)= 2j(2N-2j)-1$ and therefore $\CS(\co)=\sD_{2j}^o(N)$.

{\sf (iii)} These $SL_N$-orbits are spherical and for them $\g({\ge}3)=\{0\}$. By Theorem~\ref{thm:p94},
this means that $\es=\es_\star$.  If  $\blb=(2^r,1^{N-2r})$, then
\[
  \eus D(\co_\blb)=(\underbrace{0 \cdots 0}_{r-1}\,1\,\underbrace{0\cdots 0}_{N-2r-1}\,1\,\underbrace{0 \cdots 0}_{r-1}) 
\]
and $\g(0)=\mathfrak{sl}_r\oplus\mathfrak{sl}_{N-2r}\oplus\mathfrak{sl}_r\oplus\te_2$. Here
$\g(2)=\esi_1\esi_2\vp_1\vp''$ (i.e., $\mathfrak{sl}_{N-2r}$ acts trivially on $\g(2)$) and
$\ka=\Delta(\mathfrak{sl}_r)\oplus\mathfrak{sl}_{N-2r}\oplus\te_1=\Delta(\mathfrak{sl}_r)\oplus\mathfrak{gl}_{N-2r}$. Then $\es=\es_\star=\te(\mathfrak{sl}_r)\oplus\mathfrak{gl}_{N-2r}$ and
$r(\co_\blb)=\rk\,\slN-\rk\,\es=r$.

In this case, $\Gamma(\co_\blb)$ is freely generated by the dominant weights 
$\varpi_1{+}\varpi_{N-1}, \dots, \varpi_r{+}\varpi_{N-r}$, see~\cite[Sect.\,6]{p03}. Hence 
$\te_{\co_\blb}=\{\mathsf{diag}(\mu_1,\dots,\mu_r,0,\dots,0,-\mu_r,\dots,-\mu_1)\}$. 
This shows that $\CS(\co_\blb)=\ov{G{\cdot}\te_{\co_\blb}}$ belongs to $\sD_{2r}^o(N)$ and
$\sigma_3,\sigma_5,\dots,\sigma_{2r-1}$ vanish on $\CS(\co_\blb)$. Comparing dimensions, we 
obtain $\CS(\co_\blb)=\sD_{2r}^o(N)\cap \{\sigma_3=\dots =\sigma_{2r-1}=0\}$.
\end{proof}

\subsection{$\g=\spn$}   \label{subs:sp} 
For explicit computations with $(\BZ,h)$-gradings, we use the following model of $\spn$.
Let $\cI_n$ be the $n\times n$ matrix with 1's along the antidiagonal and
$\cJ=\begin{pmatrix} 0 & \cI_n \\ -\cI_n & 0 \end{pmatrix}$. Then
$\spn=\{M\in\mathfrak{gl}_{2n} \mid M^t\cJ+\cJ M=0\}$. For this model, $\be(\spn)$ (resp. $\te(\spn)$) is 
the set of upper-triangular (resp. diagonal) matrices in $\spn$. We also identify the determinantal varieties
$\sym_m (2n)$ with subvarieties of $\spn$ using $\cJ$, see Appendix~\ref{appendix}.

Consider the following subposets of $\N/Sp_{2n}$.
\begin{enumerate}
\item $\ups_{\sf max}=\ups_n=\{\co\mid \rank(e)\ge n \ \text{ for } e\in\co\}$;
\item $\ups_j=\{\co\mid \rank(e)=j \ \text{ for } e\in\co\}$, where $j=1,2,\dots,n-1$.
\end{enumerate}
Then  $\N_o/Sp_{2n}=\bigsqcup_{j\ge 1}^n \ups_j$. Let $\blb=(\lb_1,\dots,\lb_p)$ be a ``symplectic''
partition of $2n$. Then
\begin{itemize}
\item \ $\co_\blb\in\ups_{\sf max}$ if and only if $2n-p\ge n$. The only minimal element of $\ups_{\sf max}$ corresponds to the partition $(2^n)$.
\item \ $\co_\blb\in\ups_{j}$ if and only if $2n-p=j\le n-1$. Therefore,
each $\ups_j$ ($j \le n-1$) has a unique minimal and unique maximal element. Namely,
$\blb_{\sf min}(j)=(2^j, 1^{2n-2j})$ and $\blb_{\sf max}(j)=\begin{cases}
(j,2, 1^{2n-2-j}), & \ \text{ if $j$ is even,} \\
(j+1,1^{2n-1-j}), & \ \text{ if $j$ is odd.}  
\end{cases}$
\end{itemize}

\begin{thm}  \label{thm:sp}
Let $e\in\N(\spn)$ and $\co=Sp_{2n}{\cdot}e$.
\begin{itemize}
\item[\sf (i)] \ If\/ $\co\in\ups_{\sf max}$, then $\es_\star=\{0\}$, 
$r(\co)=\rk\,\spn=n$, and $\CS(\co)=\spn$;
\item[\sf (ii)] \ If\/ $\co\in\ups_j$ \ for $j<n$, then $\es_\star=\mathfrak{sp}_{2n-2j}$, 
$r(\co)=j$, and $\CS(\co)\simeq \sym_{2j}(2n)$.
\end{itemize}
\end{thm}
\begin{proof}
{\sf (i)}  Since $\co_{(2^n)}=\osp$, we already know that $r(\co_{(2^n)})=n$, see 
Example~\ref{ex:max-sph-full}.

{\sf (ii)}   For $\blb=\blb_{\sf min}(j)=(2^j,1^{2n-2j})$, we have 
$h=\mathsf{diag}(\underbrace{1,\dots,1}_{j},\underbrace{0,\dots,0}_{2n-2j},\underbrace{-1,\dots,-1}_{j})$,
\\[.6ex]
\centerline{ 
$ \eus D(\co_\blb)=(\underbrace{0 \cdots 0}_{j-1}\,1\,\underbrace{0\cdots 0\Leftarrow 0}_{n-j})$, 
\ and \ $\g(0)=\mathfrak{sl}_j\oplus\mathfrak{sp}_{2n-2j}\oplus\te_1$.}
\\ 
Here
$\g(1)=\vp_1\vp'_1\esi_1$, $\g(2)=\vp_1^2\esi_1^2$, and $\g({\ge}3)=\{0\}$. Therefore, 
$\ka=\mathfrak{so}_j\oplus\mathfrak{sp}_{2n-2j}$ and $\es=\es_\star=\mathfrak{sp}_{2n-2j}$. 
Note that this also works for $j=n$, i.e., for part {\sf (i)} above.

For $\blb=\blb_{\sf max}(j)$, one has to consider two possibilities. 

\noindent
\textbullet \ \ If $j=2m-1$ is odd, then $\blb=(2m,1^{2n-2m})$ and $2m-1< n$. In this case,
\[
   h=\mathsf{diag}(2m-1,\dots,3,1,\underbrace{0,\dots,0}_{2n-2m},-1,-3,\dots,-2m+1)  
\]
and $\eus D(\co_\blb)=(\underbrace{2 \cdots 2}_{m-1}\,1\,\underbrace{0\cdots 0\Leftarrow 0}_{n-m})$.
Then $\g(0)=\mathfrak{sp}_{2n-2m}\oplus\te_m$ and $\g(2)=\esi_1+\dots +\esi_m$, i.e., 
$\mathfrak{sp}_{2n-2m}$ acts trivially on $\g(2)$. Hence $\es=\ka=\mathfrak{sp}_{2n-2m}$.
In this case, $\g(2i-1)=\esi_i\vp_1$ for $i=1,2,\dots,m-1$, whereas $\mathfrak{sp}_{2n-2m}$ acts also trivially on $\g(2i)$ with $i\ge 2$. 
\\ \indent
Therefore, $\g({\ge} 3)=(m-1)\vp_1+ (m^2-m)\odin$ as $\es$-module. Thus, a generic isotropy 
subalgebra for $(\es: \g({\ge} 3)\oplus \g({\ge} 3)^*)$ is equal to 
$\mathfrak{sp}_{2n-4m+2}=\mathfrak{sp}_{2n-2j}$, as required.

\noindent
\textbullet \ \ If $j=2m$ is even, then $\blb=(2m,2,1^{2n-2m-2})$ and $2m<n$. Then
\[
   h=\mathsf{diag}(2m-1,\dots,3,1,1,\underbrace{0,\dots,0}_{2n-2m-2},-1,-1,-3,\dots,-2m+1)   
\]
and $\eus D(\co_\blb)=(\underbrace{2 \cdots 2}_{m-1}\,11\,\underbrace{0\cdots 0\Leftarrow 0}_{n-m-1})$.
We omit further details.

{\bf --} \ Since $\es_\star=\mathfrak{sp}_{2n-2j}$ for all $\co\in\ups_j$, we obtain $r(\co)=j$ 
and $\dim\CS(\co)=\dim\g-\dim\es_\star$. On the other hand, $\rank(e)=j$ for all $\co\in\ups_j$ and hence 
$\CS(\co)\subset \sym_{2j}(2n)$. Here
$\dim\g-\dim\es_\star=\dim \sym_{2j}(2n)= j(4n-2j+1)$ and therefore $\CS(\co)=\sym_{2j}(2n)$.
\end{proof}

\subsection{$\g=\soN$}   
\label{subs:so}
Since $\mathfrak{so}_5\simeq\mathfrak{sp}_4$ and $\mathfrak{so}_6\simeq\mathfrak{sl}_4$,
we assume that $N\ge 7$. For explicit computations with $(\BZ,h)$-gradings, we use the model of $\soN$
as the set of skew-symmetric matrices w.r.t.~the antidiagonal. Accordingly, the determinantal variety
$\sk_{2p}(N)\subset\soN$ consists of matrices of the same form, see Appendix~\ref{appendix}.

Let us consider the following subposets of $\N/SO_{N}$.
\begin{enumerate}
\item $\ups_{\sf max}=\{\co\mid \rank(e)\ge [N/2] \ \ \& \ \text{ $e^2\ne 0$ for $e\in\co$}\}$;
\item $\ups_m=\{\co\mid \rank(e)=2m< [N/2] \ \ \& \ \text{ $e^2\ne 0$ for $e\in\co$}\}$ for $m\ge 1$;
\item $\ups_{(1,m)}=\{\co_{(2^{2m}, 1^{N-4m})}\}$ for $m=1,2,\dots, [N/4]$.
\end{enumerate}
Then each $\ups_{(1,m)}$ consists of a single spherical $SO_N$-orbit and 
\[
    {\N}_o/SO_N=\ups_{\sf max}\sqcup (\bigsqcup_{m\ge 1}\ups_m)\sqcup (\bigsqcup_{m=1}^{[N/4]}\ups_{(1,m)}) .
\]
Let $\blb=(\lb_1,\dots,\lb_p)$ be an ``orthogonal'' partition of $N$. Then $N-p$ is even and
\begin{itemize}
\item \ $\co_\blb\in\ups_{\sf max}$ if and only if $N-p\ge N/2$ and $\lb_1\ge 3$. Hence the minimal elements of $\ups_{\sf max}$ correspond to the partitions
\begin{itemize}
\item $(3,2^{2m-2},1)\sim\osp$,  if $N=4m$, i.e., $\g=\eus D_{2m}$ with $m\ge 2$;
\item $(3,2^{2m-2},1,1)\sim\osp$,  if $N=4m+1$, i.e., $\g=\eus B_{2m}$ with $m\ge 2$;
\item $(3,3,2^{2m-2})$,  if $N=4m+2$, i.e., $\g=\eus D_{2m+1}$ with $m\ge 2$;
\item $(3,2^{2m})\sim\osp$,  if $N=4m+3$, i.e., $\g=\eus B_{2m+1}$ with $m\ge 1$;
\end{itemize}
\item \ $\co_\blb\in\ups_{m}$ if and only if $N-p=2m<[N/2]$ and $\lb_1\ge 3$. Therefore,
each $\ups_m$ has a unique minimal and unique maximal element, the corresponding partitions being
$\blb_{\sf min}(2m)=(3, 2^{2m-2}, 1^{N+1-4m})$ and $\blb_{\sf max}(2m)=(2m+1, 1^{N-2m-1})$. Note
that the condition $2m<[N/2]$ means that $N+1-4m\ge 3$.
\end{itemize}

\begin{thm}   \label{thm:so}
Let $e\in\N(\soN)$ and $\co=SO_N{\cdot}e$.
\begin{itemize}
\item[\sf (i)] \ If\/ $\co\in\ups_{\sf max}$, then $\es_\star=\{0\}$, $r(\co)=\rk\,\soN=[N/2]$, and 
$\CS(\co)=\soN$;
\item[\sf (ii)] \ If\/ $\co\in\ups_m$, then $\es_\star=\mathfrak{so}_{N-4m}$, $r(\co)=2m$, and 
$\CS(\co)=\sk_{4m}(N)$;
\item[\sf (iii)] \ If\/ $\blb=(2^{2m}, 1^{N-4m})$, then $\co_\blb$ is spherical and $4m\le N$.
Here $\es_\star=\mathfrak{so}_{N-4m}\oplus (\tri)^m$, $r(\co)=\frac{1}{2}\rank(e)=m$, and
$\CS(\co_{\blb})$ is of codimension $3m$ in $\sk_{4m}(N)$.
\end{itemize}
\end{thm}
\begin{proof}
{\sf (i)} If $N\ne 4m+2$, then $-1\in \eus W$ and $\osp$ is the minimal element of $\ups_{\sf max}$.
Hence $\es_\star=\{0\}$ and $r(\co)=\rk\,\soN$ for all $\co\in \ups_{\sf max}$. Therefore, we have only to 
consider the non-spherical orbit $\co_\blb$ associated with $\blb=(3,3,2^{2m-2})$ for $\eus D_{2m+1}$.
Here \\[.6ex]
\centerline{
$h=\mathsf{diag}(2,2,\underbrace{1,\dots,1}_{2m-2},0,0,\underbrace{-1,\dots,-1}_{2m-2},-2,-2)$ \ \& \ 
$\ap_i(h)=1$ for $i=2,2m,2m+1$.} 
\\[.6ex]
Therefore, $\g(0)=\mathfrak{sl}_{2m-2}\oplus\tri\oplus\te_3$ and
$\g(1)=\esi_1\vp_1\vp'+\esi_2\vp_{2m-1}+\esi_3\vp_{2m-1}$. Then
\begin{itemize}
\item $\g(2)=\esi_1\esi_2\vp'+\esi_1\esi_3\vp'+ \esi_2\esi_3\vp_{2m-2}$;
\item $\g(3)=\esi_1\esi_2\esi_3\vp_{2m-1}\vp'$;
\item $\g(4)=\esi_1^2\esi_2\esi_3$.
\end{itemize}
Here $\ka=\mathfrak{sp}_{2m-2}\oplus\mathfrak{so}_2=\mathfrak{sp}_{2m-2}\oplus\te_1$ and
$\es=\mathsf{g.i.s.}(\ka:\g(2))=(\tri)^{m-1}\subset \mathfrak{sp}_{2m-2}$. Therefore,
$\g(3)\vert_\es=2(\vp^{(1)}+\dots +\vp^{(m-1)})$ and hence $\es_\star=\{0\}$.

{\sf (ii)} As usual, it suffices to check that $\es_\star=\mathfrak{so}_{N-4m}$ for both
$\co_{\sf min}(2m)$ and $\co_{\sf max}(2m)$. 

{\bf --} \ The orbit $\co_{\sf min}(2m)\sim (3,2^{2m-2},1^{N+1-4m})$
is spherical,  and related calculations for it can be found in~\cite[Sect.\,4]{p94}.  

{\bf --} \ For $\co_{\sf max}(2m)\sim (2m+1, 1^{N-2m-1})$, the dominant characteristic is \\
$h=\mathsf{diag}(2m,\dots,4,2,\underbrace{0,\dots,0}_{N-2m},-2,-4,\dots,-2m)$.  Hence
$\ap_i(h)=2$ for $i\le m$ and $\ap_i(h)=0$ for $i > m$. Here $\g(0)=\mathfrak{so}_{N-2m}\oplus\te_m$ and 
$\g(2)=\esi_+\dots+\esi_{m-1}+\esi_m\vp_1$. Therefore, $\ka=\mathfrak{so}_{N-2m-1}$ and
$\es=\mathfrak{so}_{N-2m-2}=\mathfrak{so}(\BV)$. A direct calculation shows that 
$\g({\ge}3)\vert_\es=(m-1)\BV+(m^2-1)\odin$. Thus,
\[
  \es_\star=\mathsf{g.i.s.}\bigl(SO(\BV): 2(m{-}1)\BV\bigr)=\mathfrak{so}_{\dim\BV-2m+2}=\mathfrak{so}_{N-4m} .
\]

{\bf --} \ Since $\es_\star=\mathfrak{so}_{N-4m}$ for all $\co\in\ups_j$, we obtain $r(\co)=2m$ and 
$\dim\CS(\co)=\dim\g-\dim\es_\star$. On the other hand, $\rank(e)=2m$ for all $\co\in\ups_j$ and hence 
$\CS(\co)\subset \sk_{4m}(N)$. Here
$\dim\g-\dim\es_\star=\dim \sk_{4m}(N)= 2m(2N-4m-1)$ and therefore $\CS(\co)=\sk_{4m}(N)$.

{\sf (iii)} For $\blb=(2^{2m},1^{N-4m})$, the orbit $\co_\blb$ is spherical and
\\[.6ex]
\centerline{
$h=\mathsf{diag}(\underbrace{1,\dots,1}_{2m},\underbrace{0,\dots,0}_{N-4m},\underbrace{-1,\dots,-1}_{2m})$.}
\\[.6ex]
Hence $\g(0)=\mathfrak{sl}_{2m}\oplus\mathfrak{so}_{N-4m}\oplus\te_1$,
$\g(2)=\esi^2\vp_2$, and $\g({\ge}3)=\{0\}$. Note that $\mathfrak{so}_{N-4m}$ acts trivially on $\g(2)$.
Then $\ka=\mathfrak{sp}_{2m}\oplus\mathfrak{so}_{N-4m}$ and 
$\es_\star=\es=(\tri)^m\oplus\mathfrak{so}_{N-4m}$. The rest is explained in Example~\ref{ex:ex}.
\end{proof}

\section{The defective nilpotent orbits in the exceptional Lie algebras}
\label{sect:defective-exc}
\noindent
Here we describe the nilpotent orbits $\co$ in the exceptional Lie algebras such that $\Sec(\ov{\BP\co})$ 
is defective. By Corollary~\ref{cor:pro-defekt}, this is equivalent to that $r(\co)< \rk\,\g$. By 
Corollary~\ref{cor:sec=g}, another equivalent condition is that $\CS(\co)\ne\g$. In this case, the generic 
stabiliser $S_\star$ for the diagonal action $(G:\co\times\co)$ must be infinite and we also point out
$\es_\star$ for all defective orbits.

\begin{thm}   \label{thm:CS-except}
The defective nilpotent orbits $\co$ in the exceptional Lie algebras are:
\begin{itemize}
\item \ $\mathsf A_1, 2\mathsf A_1, 3\mathsf A_1, \mathsf A_2$ \ -- \ for $\eus E_6$ and $\eus E_8$;
\item \ $\mathsf A_1, 2\mathsf A_1, (3\mathsf A_1)', (3\mathsf A_1)'', \mathsf A_2$ \ -- \ for $\eus E_7$;
\item \ $\mathsf A_1, \tilde{\mathsf A}_1$ \ -- \ for $\eus F_4$;
\item \ $\mathsf A_1$  \ -- \ for $\eus G_2$.
\end{itemize}
The orbits $\mathsf A_2$ for $\g=\eus E_n$ $(n=6,7,8)$ are of complexity~$2$, whereas all other orbits 
are spherical.
\end{thm}
\begin{proof}
1) \ Let us prove that $\CS(\co)\ne\g$ for all orbits in the theorem. 
\\ \indent
The orbits $\co=m\mathsf A_1$ are spherical (with or without `prime' or `tilde'). 
For all spherical nilpotent orbits, the numbers $r(\co)$  can be found in~\cite[Sect.\,6]{p03}, where orbits 
of the exceptional groups are represented by the weighted Dynkin diagrams ({\it\bfseries wDd}). For the 
cases above, one has $r(\co)<\rk\,\g$, see Table~\ref{table:exc-defective}.

\noindent
If the highest root of $\g$ is fundamental, then the orbit $\mathsf A_2$ has a uniform description.  
The {\it\bfseries wDd} of $\mathsf A_2$ is ``twice'' the {\it\bfseries wDd} of the minimal orbit 
$\omin=\mathsf A_1$. That is, the {\it\bfseries wDd} for $\mathsf A_1$ has a unique nonzero mark `1', 
which has to be replaced with the mark `2'. Using this fact and Theorem~\ref{thm:p94}, one proves
that here we always have $c(\mathsf A_2)=2$ and $r(\mathsf A_2)=4$.  

2) \ Recall that $\osp$ is the maximal spherical nilpotent orbit in $\g$. 
\\ \indent
If $\g\ne\eus E_6$, then $-1\in \eus W(\g)$, $r(\osp)=\rk\,\g$, and $\CS(\osp)=\g$, see 
Example~\ref{ex:max-sph-full}. Therefore, if $\co\in\N/G$ and $\bco\supset\co_{\sf sph}$, then 
$\CS(\co)=\g$ as well. Looking at the Hasse diagrams of $(\N/G,\curle)$~~\cite[IV.2]{spalt} and discarding 
all such non-defective orbits, we get only the orbits pointed out above. 

If $\g=\eus E_6$, then $\osp=3\mathsf A_1$ and $r(\osp)=4$. The only orbit covering $3\mathsf A_1$ in
$\N/\GR{E}{6}$ is $\mathsf A_2$, where $r(\mathsf A_2)=4$ as well. The only orbit covering 
$\mathsf A_2$ is $\mathsf A_2{+}\mathsf A_1$, where the {\it\bfseries wDd} is
\[
\eus D(\mathsf A_2{+}\mathsf A_1)=\left(\text{\begin{E6}{1}{0}{0}{0}{1}{1}\end{E6}}\right). 
\]
As in Section~\ref{sect:comput-c-r}, we use Theorem~\ref{thm:p94} and the related $(\BZ,h)$-grading to compute $\es_\star$.
Here $\g(0)=\mathfrak{sl}_4\oplus\te_3$,
$\g(2)=\esi_1\esi_2\vp_1+\esi_1\esi_3+\esi_2\esi_3\vp_3$, \ 
$\g(3)=\esi_1\esi_2\esi_3\vp_2$, and 
$\g(4)=\esi_1\esi_2^2\esi_3$.
Then $\ka=\sltri\oplus\te_1\simeq\mathfrak{gl}_3$ and $\g(2)\vert_\ka=\mu^4\tvp_1+\mu^{-4}\tvp_2+3\odin$, where $\tvp_1$ and $\tvp_2$ are the fundamental weights of $\sltri$ and $\mu$ is the basic character of the centre of $\ka$.
Then $\es\simeq\mathfrak{gl}_2$ and if $\vp$ is the only fundamental weight of $\tri$, then
$\g(3)\vert_\es=\mu^6+2\vp+\mu^{-6}$ and $\g(4)\vert_\es=\odin$. Hence
$\es_\star=\mathsf{g.i.s.}\bigl(\es: \g(3)\oplus \g(3)^*\bigr)=\{0\}$.
Thus, $r(\mathsf A_2{+}\mathsf A_1)=6$ and 
$\mathsf A_2$ is maximal among the orbits $\co\in \N/\GR{E}{6}$ such that $\CS(\co)\ne\eus E_6$.
\end{proof}

\begin{rmk}   \label{rem:CS-except}
For $G=\GR{F}{4}$, one has $\osp=\mathsf A_1{+}\tilde{\mathsf A}_1$, and the orbits covering $\osp$ in 
$\N(\eus F_4)$ are $\mathsf A_2$ and $\tilde{\mathsf A}_2$. These three orbits have rank `4', hence they 
do not occur in Theorem~\ref{thm:CS-except}.
\end{rmk}

\begin{rmk}
The Hasse diagrams of $(\N(\g),\curle)$ show that, for $\GR{E}{7}$, there are two maximal defective
orbits, $\mathsf A_2$ and $(3\mathsf A_1)''$. In the other cases, $\mathsf A_2$ is the only maximal 
defective orbit.
\end{rmk}
In Table~\ref{table:exc-defective}, we provide the relevant information on the defective orbits. Recall that
$\dim\CS(\co)=2\dim\co-2c(\co)-r(\co)$, see Theorem~\ref{thm:main-dim}.

\begin{table}[ht]
\caption{The defective orbits in the exceptional Lie algebras}
\label{table:exc-defective}
\begin{tabular}{|cccccc||cccccc|}
$\g$ & $\co$ & $\dim\co$ & $r(\co)$ & $c(\co)$ & $\es_\star$ & $\g$ & $\co$ & $\dim\co$ & 
$r(\co)$ & $c(\co)$ & $\es_\star$ 
\\ \hline
$\eus E_6$ & $\mathsf A_1$ & 22 & 1 & 0 & $\eus A_5$ & $\eus E_7$ & $\mathsf A_1$ & 34 & 1 & 0 & 
$\eus D_6$ \\
 & $2\mathsf A_1$ & 32 & 2 & 0 & $\eus A_3{+}\te_1$ &  & $2\mathsf A_1$ & 52 & 2 & 0 & $\eus D_4{+}\eus A_1$ \\
 & $3\mathsf A_1$ & 40 & 4 & 0 & $\te_2$ &  & $(3\mathsf A_1)''$ & 54 & 3 & 0 & $\eus D_4$ \\
 & $\mathsf A_2$ & 42 & 4 & 2 & $\te_2$ &  & $(3\mathsf A_1)'$ & 64 & 4 & 0 & $(\eus A_1)^3$ \\ \cline{1-6}
$\eus E_8$ & $\mathsf A_1$ & 58 & 1 & 0 & $\eus E_7$ &  & $\mathsf A_2$ & 66 & 4 & 2 & $(\eus A_1)^3$ \\ \cline{7-12}
 & $2\mathsf A_1$ & 92 & 2 & 0 & $\eus D_6$ & $\eus F_4$ & $\mathsf A_1$ & 16 & 1 & 0 & $\eus C_3$ \\
 & $3\mathsf A_1$ & 112 & 4 & 0 & $\eus D_4$ & & $\tilde{\mathsf A}_1$ & 22 & 2 & 0 & $\eus C_2$ \\ \cline{7-12}
 & $\mathsf A_2$ & 114 & 4 & 2 & $\eus D_4$ & $\eus G_2$ & $\mathsf A_1$ & 6 & 1 & 0 & $\eus A_1$ \\ \hline
\end{tabular}
\end{table}
For all other orbits, we have $\es_\star=\{0\}$, $r(\co)=\rk\,\g$, and $c(\co)=\dim\co-\dim\be$.

Using the rank group of a spherical orbit~\cite[Sect\,6]{p03}, we can explicitly describe the toral algebra 
$\te_\co$ for all spherical orbits, cf. Example~\ref{ex:ex}. By Proposition~\ref{prop:r=r'},
for the non-spherical orbit $\mathsf A_2$, the conical secant variety coincide with that 
for $3\mathsf A_1$ ($(3\mathsf A_1)'$ for $\eus E_7$). This provides a rather explicit presentation of
$\CS(\co)$ for all defective orbits. Furthermore, in some cases, the equations defining $\CS(\co)$
can be given.

\begin{ex}   \label{ex:E6}
For the orbits $3\mathsf A_1$ and $\mathsf A_2$ in $\g=\eus E_6$, the common variety
$\CS(\co)$ is of codimension~2 in $\g$. Here one can prove that it is a complete intersection.
Namely, for the adjoint representation of $\GR{E}{6}$, the degrees of the basic invariants are
$2,5,6,8,9,12$. Hence there are two basic invariants of odd degrees, say $F_5$ and $F_9$. Then 
\[
     \CS(3\mathsf A_1)=\CS(\mathsf A_2)=\{x\in \eus E_6 \mid F_5(x)=F_9(x)=0\} .
\]
\end{ex}

\section{Some explicit data for $\co$ and $\CS(\co)$}
\label{sect:data}

\noindent For all nilpotent orbits $\co$, we explicitly describe the canonical generic isotropy subalgebra
$\ess$ (in other words, the canonical embedding $\ess\hookrightarrow \g$) 
and the dense orbit $\CS(\co)\cap\N$.

\subsection{The canonical embedding of $\es_\star$}
\label{subs:canon-emb}

As always, $\ess$ is a generic isotropy subalgebra for the diagonal action $(G:\co\times\co)$. 
In Tables~\ref{table:canon-sl-sp}, \ref{table:canon-so}, and \ref{table:canon-exc}, we describe the canonical embeddings of 
$\ess$ for all $\co\in \N/G$. In view of results of Section~\ref{sect:comput-c-r}, it suffices to point 
out the canonical embedding of $\ess$ for all subposets $\{\ups_\xi\}_{\xi\in\Xi}$ in $\N_o/G$ except 
$\ups_{\sf max}$, where $\ess=\{0\}$.

If $\ess\hookrightarrow \g$ is the canonical embedding, then properties ($\eus P_2$)--($\eus P_3$)
in Section~\ref{subs:twist} show that $\Pi(\ess)$ (the set of simple roots of $\ess$ 
w.r.t.~$\be_\star=\be\cap\ess$ and $\te_\star=\te\cap\ess$) is a subset of $\Pi=\Pi(\g)$. In some cases, there is a unique 
embedding of the Dynkin diagram of $\ess$ into the Dynkin diagram of $\g$, denoted $\cD(\g)$. This 
already yields the canonical embedding of $[\ess,\ess]$ in $\g$ without much ado.  Otherwise, one has to 
carefully look at the canonical embeddings of stabilisers in all steps of inductive computations of $\ess$.

In the tables below, the canonical embedding is represented by an ``enhanced'' Dynkin diagram. The 
black nodes of the Dynkin diagram $\cD(\g)$ represents the simple roots of $\ess$, while the embedding 
of the centre of $\ess$, $\z(\ess)$, is depicted by arcs connecting certain pairs of  white nodes. The arc 
connecting nodes $i$ and $j$ in $\cD(\g)$ gives rise to the element $\ap_i-\ap_j\in\te^*\simeq\te$. All 
these elements form a basis for $\z(\ess)$. Thus, $\dim\z(\ess)$ equals the number of arcs and 
$\rk\,[\ess,\ess]$ is the number of black nodes. Yet another hint is that the involution $\vth$ 
(cf.~Section~\ref{subs:twist}) induces an automorphism of $\cD(\g)$, and the enhanced diagram should 
to be preserved by this automorphism. This automorphism is nontrivial (and hence of order $2$) for
$\sln$ $(n\ge 3)$, $\mathfrak{so}_{4m+2}$, and $\eus E_6$.

Enhanced Dynkin diagrams resemble the Satake diagrams associated with the isotropy 
representations of the symmetric spaces or, equivalently, the real forms of complex simple Lie algebras,
cf.~\cite[Table\,4]{t41}. Such diagrams have already been used in the setting
of canonical generic stabilisers related to $G$-actions on affine double cones, see~\cite[Chap.\,6]{p99}.

\begin{table}[ht]
\caption{The canonical embedding of $\ess$ for $\slN$ and $\spn$}
\label{table:canon-sl-sp}
\begin{tabular}{|ccccc|}
\hline
$\g$ & poset & cond. & $\ess$ & enhanced $\cD(\g)$ \\ \hline
$\slN$ & $\ups_j$ &  $2{\le} j{<} \frac{N}{2}$  & $\mathfrak{gl}_{N-2j}$   &
\begin{picture}(190,25)(0,7)
\setlength{\unitlength}{0.016in}
\multiput(10,8)(30,0){2}{\circle{5}}
\multiput(55,8)(60,0){2}{\circle{5}}
\multiput(70,8)(30,0){2}{\circle*{5}}
\multiput(130,8)(30,0){2}{\circle{5}}
\multiput(43,8)(15,0){2}{\line(1,0){9}}
\multiput(103,8)(15,0){2}{\line(1,0){9}}
\multiput(13,8)(60,0){3}{\line(1,0){4}}
\multiput(33,8)(60,0){3}{\line(1,0){4}}
\multiput(19.5,5)(60,0){3}{$\cdots$}   
{\color{blue}
\put(85,14){\oval(60,18)[t] }
\multiput(55,14)(60,0){2}{\vector(0,-1){3}}
}
\put(66,3){$\underbrace{\mbox{\hspace{38\unitlength}}}_{N-1-2j}$}
\put(6,3){$\underbrace{\mbox{\hspace{53\unitlength}}}_{j}$}
\put(111,3){$\underbrace{\mbox{\hspace{53\unitlength}}}_{j}$}
\end{picture}
\\
 &  $\ups_{1,r}$ & $r<\frac{N}{2}$ & $\mathfrak{gl}_{N-2r}{\times}\te_{r-1}$ &  
 \begin{picture}(190,50)(0,7)
\setlength{\unitlength}{0.016in}
\multiput(10,8)(30,0){2}{\circle{5}}
\multiput(55,8)(60,0){2}{\circle{5}}
\multiput(70,8)(30,0){2}{\circle*{5}}
\multiput(130,8)(30,0){2}{\circle{5}}
\multiput(43,8)(15,0){2}{\line(1,0){9}}
\multiput(103,8)(15,0){2}{\line(1,0){9}}
\multiput(13,8)(60,0){3}{\line(1,0){4}}
\multiput(33,8)(60,0){3}{\line(1,0){4}}
\multiput(19.5,5)(60,0){3}{$\cdots$}   
{\color{blue}
\multiput(19.5,13)(120,0){2}{$\cdots$} 
\put(85,14){\oval(150,30)[t]}
\put(85,14){\oval(90,25)[t]}
\put(85,14){\oval(60,18)[t]}
\multiput(55,14)(60,0){2}{\vector(0,-1){3}}
\multiput(40,14)(90,0){2}{\vector(0,-1){3}}
\multiput(10,14)(150,0){2}{\vector(0,-1){3}}
}
\put(66,3){$\underbrace{\mbox{\hspace{38\unitlength}}}_{N-1-2r}$}
\put(6,3){$\underbrace{\mbox{\hspace{53\unitlength}}}_{r}$}
\put(111,3){$\underbrace{\mbox{\hspace{53\unitlength}}}_{r}$}
\end{picture}
 \\
&  $\ups_{1,r}$ & $r=\frac{N}{2}$ & $\te_{r-1}$ &   
\begin{picture}(125,45)(5,7)
\setlength{\unitlength}{0.016in}
\multiput(10,8)(90,0){2}{\circle{5}}
\multiput(40,8)(15,0){3}{\circle{5}}
\multiput(43,8)(15,0){2}{\line(1,0){9}}
\multiput(13,8)(60,0){2}{\line(1,0){4}}
\multiput(33,8)(60,0){2}{\line(1,0){4}}
\multiput(19.5,5)(60,0){2}{$\cdots$}   
{\color{blue}
\multiput(19.5,13)(60,0){2}{$\cdots$} 
\put(55,14){\oval(90,24)[t]}
\put(55,14){\oval(30,16)[t] }
\multiput(10,14)(90,0){2}{\vector(0,-1){3}}
\multiput(40,14)(30,0){2}{\vector(0,-1){3}}
}
\put(66,3){$\underbrace{\mbox{\hspace{38\unitlength}}}_{r-1}$}
\put(6,3){$\underbrace{\mbox{\hspace{38\unitlength}}}_{r-1}$}
\end{picture} 
\\  
$\spn$ & $\ups_j$ &  $1{\le} j{\le} n{-}1$  & $\mathfrak{sp}_{2n-2j}$   &
\begin{picture}(95,33)(20,5)
\setlength{\unitlength}{0.016in}
\multiput(10,8)(30,0){2}{\circle{5}}
\multiput(55,8)(30,0){2}{\circle*{5}}
\put(100,8){\circle*{5}}
\multiput(89.3,7)(0,2){2}{\line(1,0){12.3}}
\multiput(13,8)(45,0){2}{\line(1,0){4}}
\multiput(33,8)(45,0){2}{\line(1,0){4}}
\multiput(19.5,5)(45,0){2}{$\cdots$}   
\put(43,8){\line(1,0){14}}
\put(51,3){$\underbrace{\mbox{\hspace{53\unitlength}}}_{n-j}$}
\put(87,5.9){\footnotesize $<$} 
\end{picture} 
\\
 & & & & \\ \hline
\end{tabular}
\end{table}

\noindent 
The enhanced diagrams in Table~\ref{table:canon-sl-sp} show that, for  item~1,
the space $\te_\co=\te(\ess)^\perp$ consists of the following diagonal $N\times N$ matrices:
\[
     \te_\co=\{\textsf{diag}(a_1,\dots,a_{j-1},\nu_j,\underbrace{c,\dots,c}_{N-2j},-\nu_j, b_{j-1}\dots, b_1)\}\cap\slN .
\]
Whereas for items~2 and 3, one has 
$\te_\co=\{\textsf{diag}(\nu_1,\dots,\nu_r,\underbrace{0,\dots,0}_{N-2r},-\nu_r,\dots, \nu_1)\}$.

In the last row of Table~\ref{table:canon-so} for $\sone$, the poset $\ups_{1,m}$ represents actually two 
very even orbits. The enhanced diagram for the second orbit is obtain by swapping two last nodes, cf. Example~\ref{ex:ex}. The case of $\sono$ appears to be simpler, because there are no very even orbits
and $\ess$ is always semisimple.

\begin{table}[htb]
\caption{The canonical embedding of $\ess$ for $\soN$}
\label{table:canon-so}
\begin{tabular}{|ccccc|}
\hline
$\g$ & poset & cond. & $\ess$ & enhanced $\cD(\g)$ \\ \hline
$\sone$ & $\ups_m$ &  $n{-}2m{\ge}2$  & $\mathfrak{so}_{2n-4m}$   &
\begin{picture}(125,25)(0,5)
\setlength{\unitlength}{0.016in}
\multiput(10,8)(30,0){2}{\circle{5}}
\multiput(55,8)(30,0){2}{\circle*{5}}
\put(43,8){\line(1,0){9}}
\multiput(13,8)(45,0){2}{\line(1,0){4}}
\multiput(33,8)(45,0){2}{\line(1,0){4}}
\put(100,0){\circle*{5}}
\put(100,16){\circle*{5}}
\put(88,10){\line(2,1){9}}
\put(88,6){\line(2,-1){9}}
\multiput(19.5,5)(45,0){2}{$\cdots$}   
\put(6,3){$\underbrace{\mbox{\hspace{38\unitlength}}}_{2m}$}
\put(51,-2){$\underbrace{\mbox{\hspace{53\unitlength}}}_{n-2m}$}
\end{picture}
\\
 & $\ups_m$ &  $n{-}2m{=}1$  & $\mathfrak{so}_{2}=\te_1$   &
\begin{picture}(105,40)(30,5)
\setlength{\unitlength}{0.016in}
\put(40,8){\circle{5}}
\multiput(55,8)(30,0){2}{\circle{5}}
\put(43,8){\line(1,0){9}}
\put(58,8){\line(1,0){4}}
\put(78,8){\line(1,0){4}}
\put(100,0){\circle{5}}
\put(100,16){\circle{5}}
\put(88,10){\line(2,1){9}}
\put(88,6){\line(2,-1){9}}
\put(64.5,5){$\cdots$}   
{\color{blue}\put(107,8){\oval(16,16)[r]}}
{\color{blue}\put(103,16){\vector(-1,0){2}}}
{\color{blue}\put(100.5,0){\vector(-1,0){2}}}
\end{picture}
\\
 & $\ups_{1,m}$ &  $n{-}2m{\ge}2$  & $(\tri)^m\oplus\mathfrak{so}_{2n-4m}$   &
\begin{picture}(160,30)(0,5)
\setlength{\unitlength}{0.016in}
\multiput(10,8)(45,0){2}{\circle*{5}}
\multiput(25,8)(45,0){2}{\circle{5}}
\multiput(85,8)(30,0){2}{\circle*{5}}
\put(13,8){\line(1,0){9}}
\multiput(58,8)(15,0){2}{\line(1,0){9}}
\multiput(28,8)(60,0){2}{\line(1,0){4}}
\multiput(48,8)(60,0){2}{\line(1,0){4}}
\put(130,0){\circle*{5}}
\put(130,16){\circle*{5}}
\put(118,10){\line(2,1){9}}
\put(118,6){\line(2,-1){9}}
\multiput(34.5,5)(60,0){2}{$\cdots$}   
\put(6,3){$\underbrace{\mbox{\hspace{68\unitlength}}}_{2m}$}
\put(81,-2){$\underbrace{\mbox{\hspace{53\unitlength}}}_{n-2m}$}
\end{picture}
\\
 & $\ups_{1,m}$ &  $n{-}2m{=}1$  & $(\tri)^m\oplus\mathfrak{so}_{2}$   &
\begin{picture}(105,40)(30,5)
\setlength{\unitlength}{0.016in}
\put(55,8){\circle{5}}
\multiput(40,8)(45,0){2}{\circle*{5}}
\put(43,8){\line(1,0){9}}
\put(58,8){\line(1,0){4}}
\put(78,8){\line(1,0){4}}
\put(100,0){\circle{5}}
\put(100,16){\circle{5}}
\put(88,10){\line(2,1){9}}
\put(88,6){\line(2,-1){9}}
\put(64.5,5){$\cdots$}   
{\color{blue}\put(107,8){\oval(16,16)[r]}}
{\color{blue}\put(103,16){\vector(-1,0){2}}}
{\color{blue}\put(100.5,0){\vector(-1,0){2}}}
\put(32,3){$\underbrace{\mbox{\hspace{53\unitlength}}}_{2m-1}$}
\end{picture}
\\
 & $\ups_{1,m}$ &  $n{=}2m$  & $(\tri)^m$   &
\begin{picture}(120,37)(30,5)
\setlength{\unitlength}{0.016in}
\multiput(55,8)(45,0){2}{\circle{5}}
\multiput(40,8)(45,0){2}{\circle*{5}}
\multiput(43,8)(45,0){2}{\line(1,0){9}}
\put(58,8){\line(1,0){4}}
\put(78,8){\line(1,0){4}}
\put(115,0){\circle*{5}}
\put(115,16){\circle{5}}
\put(103,10){\line(2,1){9}}
\put(103,6){\line(2,-1){9}}
\put(64.5,5){$\cdots$}   
\put(36,3){$\underbrace{\mbox{\hspace{68\unitlength}}}_{2m-2}$}
\end{picture}
\\
$\sono$ & $\ups_m$ &  $n{>}2m$  & $\mathfrak{so}_{2n+1-4m}$   &
\begin{picture}(95,35)(20,5)
\setlength{\unitlength}{0.016in}
\multiput(10,8)(30,0){2}{\circle{5}}
\multiput(55,8)(30,0){2}{\circle*{5}}
\put(100,8){\circle*{5}}
\multiput(13,8)(45,0){2}{\line(1,0){4}}
\multiput(33,8)(45,0){2}{\line(1,0){4}}
\multiput(19.5,5)(45,0){2}{$\cdots$}   
\put(43,8){\line(1,0){14}}
\put(51,3){$\underbrace{\mbox{\hspace{53\unitlength}}}_{n-2m}$}
\multiput(87.5,7)(0,2){2}{\line(1,0){8.8}}
\put(92,5.9){\footnotesize $>$} 
\end{picture} 
\\
 & $\ups_{1,m}$ &  $n{\ge}2m$  & $(\tri)^m{\oplus}\mathfrak{so}_{2n+1-4m}$   & 
\begin{picture}(160,30)(0,5)
\setlength{\unitlength}{0.016in}
\multiput(10,8)(45,0){2}{\circle*{5}}
\multiput(25,8)(45,0){2}{\circle{5}}
\multiput(85,8)(30,0){2}{\circle*{5}}
\put(13,8){\line(1,0){9}}
\multiput(58,8)(15,0){2}{\line(1,0){9}}
\multiput(28,8)(60,0){2}{\line(1,0){4}}
\multiput(48,8)(60,0){2}{\line(1,0){4}}
\put(130,8){\circle*{5}}
\multiput(117.5,7)(0,2){2}{\line(1,0){8.8}}
\put(122,5.9){\footnotesize $>$} 
\multiput(34.5,5)(60,0){2}{$\cdots$}   
\put(6,3){$\underbrace{\mbox{\hspace{68\unitlength}}}_{2m}$}
\put(81,3){$\underbrace{\mbox{\hspace{53\unitlength}}}_{n-2m}$}
\end{picture}
\\
 & & & & \\ \hline
\end{tabular}
\end{table}

If $\g$ is exceptional, then non-semisimple $\ess$ occur only for $\eus E_6$. In Table~\ref{table:canon-exc}, we provide the enhanced
diagrams for the canonical embedding of $\ess$ only for the cases, where the embedding of $\Pi(\ess)$
in $\Pi$ is not {\sl a priori\/} obvious, cf. Table~\ref{table:exc-defective}.

\begin{table}[ht]
\caption{The canonical embedding of some $\ess$ in exceptional $\g$}
\label{table:canon-exc}
\begin{tabular}{|cccc|}
\hline
$\g$ & $\co$ &  $\ess$ & enhanced $\cD(\g)$ \\ \hline
$\eus E_6$  & $2\mathsf A_1$ & $\eus A_3\oplus\te_1$ &  
\begin{picture}(90,27)(25,15) 
\setlength{\unitlength}{0.016in} 
\multiput(38,18)(15,0){4}{\line(1,0){9}}
\put(65,4){\circle{5}}
\multiput(50,18)(15,0){3}{\circle*{5}}
\multiput(35,18)(60,0){2}{\circle{5}}
\put(65,7){\line(0,1){8}}
{\color{blue}\put(65,23){\oval(60,15)[t]}
\multiput(35,23.1)(60,0){2}{\vector(0,-1){2}}}
\end{picture}
\\
  & $3\mathsf A_1, \mathsf A_2$ & $\te_2$ &  
\begin{picture}(90,35)(25,15) 
\setlength{\unitlength}{0.016in} 
\multiput(38,18)(15,0){4}{\line(1,0){9}}
\put(65,5){\circle{5}}
\multiput(50,18)(15,0){3}{\circle{5}}
\multiput(35,18)(60,0){2}{\circle{5}}
\put(65,8){\line(0,1){7}}
{\color{blue}
\put(65,23){\oval(60,15)[t]}  \put(65,23){\oval(30,12)[t]}
\multiput(35,23.1)(60,0){2}{\vector(0,-1){2}}
\multiput(50,23.1)(30,0){2}{\vector(0,-1){2}}}
\end{picture}
\\
$\eus E_7$  & $(3\mathsf A_1)', \mathsf A_2$ & $\eus A_1^3$ &  
\begin{picture}(98,27)(14,11)
\setlength{\unitlength}{0.016in} 
\multiput(23,15)(15,0){5}{\line(1,0){9}} 
\put(65,2){\circle*{5}}
\multiput(65,15)(15,0){3}{\circle{5}}
\put(65,4.5){\line(0,1){7.5}}
\multiput(20,15)(30,0){2}{\circle*{5}} \put(35,15){\circle{5}}
\end{picture}
\\ 
& & &  \\ 
\hline
\end{tabular}
\end{table}

\subsection{On the orbit $\tilde\co$}
\label{subs:tilde-o}
Recall from Section~\ref{sect:main} that $\els=\ess\oplus\te_\co$ is a generic isotropy subalgebra for the
action $(G:\CS(\co))$, $\CS(\co)=\ov{G{\cdot}\te_\co}$, and $\els$ is the centraliser of $\te_\co$ in $\g$.
Then $\tilde\co$ is the dense orbit in $\CS(\co)\cap\N$ and thereby the nilpotent orbit in the Dixmier sheet 
associated with $\els$, cf. Remark~\ref{rem:Dixmier}.

If $\g$ is classical and $\els$ is given explicitly via partitions, then there is a procedure for describing 
the partition of $\tilde\co$, see e.g.~\cite{kemp}. Since $\tilde\co$ depends only on the poset $\ups_\xi$ 
containing $\co$ and $\tilde\co=\co_{\sf reg}$ for all $\co\in\ups_{\sf max}$, it suffices to determine 
$\tilde\co$ for all other subposets of $\N/G$, if $\g$ is classical, and for the orbits in Table~\ref{table:exc-defective}, if $\g$ is exceptional. Below we provide the correspondence
$(\co\in\ups_\xi) \mapsto \tilde\co$ for all $\co$ and $\g$.

{\it\bfseries For\/} $\slN$: \ 
\textbullet\quad $\co\in \ups_j\sqcup\ups_{(1,j)} \mapsto  (2j+1,1^{N-2j-1})$, \ $j<N/2$;

\phantom{{\it\bfseries For\/} $\slN$:\ }
\textbullet\quad $\co \in \ups_{(1,N/2)} \mapsto (N)$.

{\it\bfseries For\/} $\spn$: \ 
\textbullet\quad $\co\in \ups_j \mapsto  (2j,2,1^{2n-2j-2})$, \ $1\le j\le n-1$;

{\it\bfseries For\/} $\soN$: \
\textbullet\quad $\co\in \ups_m \mapsto  (4m+1, 1^{N-4m-1})$, \ $2m< [N/2]$;

\phantom{{\it\bfseries For\/} $\soN$:\ }\textbullet\quad $\co\in \ups_{(1,m)} \mapsto 
\begin{cases}
(2m,2m), & N=4m, \\
(2m+1,2m-1,1), & N=4m+1, \\
(2m+1,2m+1,1^{N-4m-2}), & N\ge 4m+2. 
\end{cases}$

{\it\bfseries For the exceptional Lie algebras:}

\textbullet\quad  $\eus E_6, \eus E_7, \eus E_8$: \quad
$\mathsf A_1\mapsto \mathsf A_2$, \ \
$\mathsf A_2\mapsto \mathsf A_4$, \ \ 
$\mathsf {3A}_1,\mathsf A_2 \mapsto \mathsf E_6$, 

\noindent where $\mathsf {3A}_1$ for $\eus E_7$ is assumed to be $(\mathsf {3A}_1)'$.

\textbullet\quad  $\eus E_7$: \quad $(\mathsf {3A}_1)''\mapsto (\mathsf {A}_5)''$.

\textbullet\quad  $\eus F_4$: \quad $\mathsf A_1\mapsto \mathsf A_2$, \ \ 
$\tilde{\mathsf A}_1\mapsto \mathsf F_4(a_3)$.

\textbullet\quad  $\eus G_2$: \quad $\mathsf A_1\mapsto \mathsf G_2(a_1)$.


\section{Some complements and applications}  \label{sect:appl}
\subsection{On the posets $\ups_j$}                  \label{subs:ups}
In Sections~\ref{subs:sl}--\ref{subs:so}, we introduced certain subposets of $\N_o/G$
for the three classical series. Surprisingly, these posets have an additional interpretation.

\begin{thm}  \label{thm:ups}  \leavevmode
\begin{itemize}
\item[{\sf (i)}] \ For\/ $\g=\slN$, the poset $\ups_j$ $(2\le j\le [(N{-}1)/2])$ is isomorphic to the poset of
{\bfseries nonzero} orbits in $\N(\mathfrak{sl}_j)$;
\item[{\sf (ii)}] \ For\/ $\g=\spn$, the poset $\ups_j$ $(1\le j\le n{-}1)$ is isomorphic to the poset of
{\bfseries all} orbits in $\N(\mathfrak{so}_j)$;
\item[{\sf (iii)}] \ For\/ $\g=\soN$, the poset $\ups_m$ $(2\le 2m < [N/2])$ is isomorphic to the poset of
{\bfseries nonzero} orbits in $\N(\mathfrak{sp}_{2m})$.
\end{itemize}
\end{thm}
\begin{proof}
On the level of partitions, all three isomorphisms are given by the mapping 
\beq  \label{eq:trunc}
      \blb=(\lb_1,\lb_2,\dots,\lb_p)\mapsto \blb'=(\lb_1-1,\lb_2-1,\dots, \lb_p-1) ,
\eeq
where the zero parts in the RHS should be ignored. By definition of these posets, the resulting $\blb'$ is 
a partition of $j$, $j$, and $2m$, respectively. The formulae for $\blb_{\sf max}(\cdot)$ and 
$\blb_{\sf min}(\cdot)$ show that in all cases $\blb_{\sf max}(\cdot)'$ is the partition of the regular nilpotent 
orbit in the corresponding smaller Lie algebra, while $\blb_{\sf min}(\cdot)'$ yields either the minimal 
nilpotent orbit in cases {\sf (i)} and {\sf (iii)} or the zero orbit in case {\sf (ii)}.

For instance, if $j$ is even in {\sf (ii)}, then $\blb=\blb_{\sf max}(j)=(j,2,1^{2n-2-j})$ and $\blb'=(j-1,1)$, 
which corresponds to $\co_{\sf reg}$ in $\mathfrak{so}_j$. It is also clear that mapping~\eqref{eq:trunc} 
change the parity of parts and thereby takes the ``symplectic'' partitions to ``orthogonal'' ones, and vice 
versa. This explains the switch $SO\leftrightarrow Sp$ in parts {\sf (ii)} and {\sf (iii)}.
\end{proof}
\begin{rmk}   \label{rem:KP82}
The passage from $\blb$ to $\blb'$ is  a special case of ``erasing columns'' in the sense of 
Kraft--Procesi~\cite{KP81,KP82}. To this end, one represents $\blb$ geometrically as Young diagram 
with rows consisting of $\lb_1,\lb_2,\dots,\lb_p$ boxes, respectively. Then~\eqref{eq:trunc} means that
one erases the leftmost column in the Young diagram. By~\cite[Prop.\,3.1]{KP81} 
and~\cite[Prop.\,3.2]{KP82}, this procedure preserves the codimension of orbits. In our case, this means
that if $\co_\blb, \co_{\bmu}\in \ups_j$ and $\co_\blb\subset \ov{\co_{\bmu}}$, then
\[
  \dim\co_{\bmu}-\dim\co_{\blb}= \dim\co_{\bmu'}-\dim\co_{\blb'} .
\]
\end{rmk}

\subsection{The growth of secant defect}   \label{subs:defect}
The secant defect of the smooth projective variety $\BP\omin$ equals~1 (Example~\ref{ex:defekt-omin}).
However, if $\co\ne\omin$, then $\ov{\BP\co}$ is not smooth. It is known that, for non-smooth projective 
varieties, the secant defect can be arbitrarily large. Our results provide a good illustration to this
thesis. We point out the series of defective orbits $\co$ in the Lie algebras $\sln$, or $\spn$, or 
$\son$ such that $\delta_{\ov{\BP\co}}$ has a quadratic growth w.r.t.~$n$.

By Theorem~\ref{thm:main-dim} and Corollary~\ref{cor:pro-defekt}, if $\co$ is defective, then 
\beq          \label{eq:defect}
  \delta_{\ov{\BP\co}}=2c(\co)+r(\co)=2\dim\co-\dim G+\dim S_\star ,
\eeq
and the Lie algebras $\ess$ are described in Theorems~\ref{thm:sl}--\ref{thm:so}.

\noindent
\textbullet \ \ 
For $2\le j\le [\frac{N-1}{2}]$, the maximal orbit in $\ups_j\subset\N_o/SL_N$ is defective and 
corresponds to $\blb_{\sf max}(j)=(j+1,1^{N-j-1})$. The complexity of $\co=\co_{\sf max}(j)$ is maximal among all orbits in $\ups_j$, and we obtain using \eqref{eq:defect} that
$\delta_{\ov{\BP\co}}=j^2+(j-1)^2$. By Theorem~\ref{thm:sl}{\sf (ii)}, $r(\co)=2j-1$ and hence
$c(\co)=(j-1)^2$.

\noindent \textbullet \ \ 
For $1\le j\le n-1$, the maximal orbit in $\ups_j\subset\N_o/Sp_{2n}$ is defective. 

{\bf --} \ If $j$ is odd, it corresponds to $\blb=\blb_{\sf max}(j)=(j+1,1^{2n-j-1})$.  If $j=2m-1$, then one obtains
$\delta_{\ov{\BP\co_\blb}}=4m^2-6m+3$, $c(\co_\blb)=2(m-1)^2$, and $r(\co_\blb)=2m-1$.

{\bf --} \ If $j$ is even,
it corresponds to $\blb=\blb_{\sf max}(j)=(j,2,1^{2n-j-2})$.  If $j=2m$, then one obtains
$\delta_{\ov{\BP\co_\blb}}=4m^2-2m$, $c(\co_\blb)=2m^2-2m$, and $r(\co_\blb)=2m$.
\\ \indent
In both cases, we use Theorem~\ref{thm:sp}{\sf (ii)} and \eqref{eq:defect}.

\noindent\textbullet \ \ 
For $1\le m < \frac{1}{2}[N/2]$, the maximal orbit in $\ups_m\subset\N_o/SO_N$ is defective and 
corresponds to $\blb_{\sf max}(j)=(2m+1,1^{N-2m-1})$. Then
$\delta_{\ov{\BP\co}}=4m^2-2m$, $c(\co)=2m^2-2m$, and $r(\co)=2m$.
\\ \indent
Here we use Theorem~\ref{thm:so}{\sf (ii)} and \eqref{eq:defect}.

\noindent \textbullet \ \ For the exceptional Lie algebras, $\delta_{\ov{\BP\co}}$ does not exceed $8$, see
Table~\ref{table:exc-defective}.

\subsection{On possible values of the complexity}
It is proved in~\cite[Chap.\,4.5]{p99} that orbits $\co$ with $c(\co)=1$ exist only in $\g=\slN$, and the only 
possibility is the partition $\blb=(3,1^{N-3})$. Using our computations for the classical algebras 
(Section~\ref{sect:comput-c-r}) and exceptional algebras (Section~\ref{sect:defective-exc}), we obtain 
more precise assertions.

\begin{lm}
If\/ $\g$ is classical and $\ups_\xi$ is any subposet of $\N/G$ described in Section~\ref{sect:comput-c-r}, then the 
complexity $c(\co)$ has the same parity for all $\co\in \ups_\xi$.
\end{lm}
\begin{proof}
Since a  Borel subalgebra of $\es_\star$, $\be(\es_\star)$, is a generic isotropy subalgebra for $(B:\co)$, 
we have $c(\co)=\dim\co-\dim\be+\dim\be(\es_\star)$. It remains to note that $\dim\co$ is even and 
$\es_\star$ is the same for all $\co\in\ups_\xi$.
\end{proof}

\begin{prop}   \label{prop:c-even1}
If\/ $\g\in\{ \eus B_n, \eus C_{n}, \eus D_{2m}\}$, then  $c(\co)$ is even for all $\co\in\N/G$. 
\end{prop}
\begin{proof}
If $\co\in\ups_{\sf max}$, then $c(\co)=\dim\co-\dim\be$ and $\dim\be$ is even in these cases.
If $\co\in\ups_j\ne\ups_{\sf max}$, then it suffices to check that $c(\co_{\sf max}(j))$ is even.
And this is done in Section~\ref{subs:defect}.
\end{proof}

\begin{prop}   \label{prop:c-even2}
If $\g$ is exceptional, then  $c(\co)$ is even for all $\co\in\N/G$.
\end{prop}
\begin{proof}
If $\co$ is non-defective,  then $c(\co)=\dim\co-\dim\be$, and $\dim\be$ is even in these cases.
The defective orbits are listed in Theorem~\ref{thm:CS-except}, and $c(\co)=0$ or $2$ for all of them.
\end{proof}

\begin{rmk}
For $\eus D_{2m+1}$, all defective orbits have an even complexity. But the complexity of non-defective
orbits is always odd (because $\dim\be=(2m+1)^2$ is odd).
\end{rmk}

\subsection{On possible values of the rank}
Using results of Section~\ref{sect:comput-c-r},  we explicitly describe the possible values of $r(\co)$ 
for the orbits in the classical Lie algebras.

\begin{prop}   \label{prop:sl-ranks}
The rank of\/ $\co\in\N/SL_N$ cannot take a value $2p$ such that
$N/2 < 2p <N{-}1$.  
\end{prop}
\begin{proof}
This readily follows from Theorem~\ref{thm:sl}.
If  $\co\in\ups_{\sf max}$, then $r(\co)=N-1$. For $\co\in\ups_j$, where $j< N/2$, we have $r(\co)=2j-1$.
If $\co\in\ups_{(1,r)}$, then $r(\co)=r\le N/2$.
\end{proof}

Likewise, using Theorem~\ref{thm:so}, we obtain
\begin{prop}    \label{prop:so-ranks}
For $\co\in\N/SO_N$, the integer $r(\co)$ cannot take a value $2p+1$ such that \\
$[N/4] < 2p+1 <[N/2]=\rk\,\son$. 
\end{prop}

For $\spn$, every $j\in [1,n]$ is the rank of some orbit $\co$.

{\bf Example.}  The forbidden values of $r(\co)$ are: \  
$6$ \ for $\mathfrak{sl}_9$; \  $6,8$ for $\mathfrak{sl}_{10}$;
$5,7$ \ for $\mathfrak{so}_{16}$.

\subsection{Higher secant varieties and joins}   \label{subs:higher}
Our approach to secant varieties of nilpotent orbits applies also to their 
higher secant varieties (with less satisfactory results).

Let $\Sec_r(X)$ denote the (higher) secant variety of $X\subset \BP^n$. Here $\Sec_1(X)=X$ and
$\Sec_2(X)=\Sec(X)$ is the usual secant variety. So, new objects occur for $r\ge 3$, 
see~\cite[Chap.\,5.1]{land}. Then $\CS_r(\co)$ stands for the affine cone over $\Sec_r(\ov{\BP\co})$.

Consider the diagonal action of $G$ on the $r$-fold product $\co^r=\co\times\ldots\times\co$. 
For $z=(x_1,\dots,x_r)\in\co^r$, consider the sum of the embedded tangent spaces 
\[
   {\mathsf T}_z=\BT_{x_1}(\co)+\dots +\BT_{x_r}(\co)\subset \g .
\]
Then $(\mathsf T_z)^\perp=\g^{x_1}\cap\ldots \cap\g^{x_r}=\g^z$. On the other hand,
Terracini's lemma shows that 
$\dim \CS_r(\co)=\dim \mathsf T_z$ for almost all $z\in\co^r$. Therefore, for any $r\ge 1$, we have
\[
     \dim \CS_r(\co)=\max_{z\in \co^r}\dim G{\cdot}z .
\]
\indent \textbullet \ \ If $r=2m$, then $\co^r\simeq \co^{m}\times(\co^{m})^*$ can be regarded as doubled 
$G$-variety. Therefore, the theory exposed in Section~\ref{subs:twist} applies to $X=\co^m$. Hence the 
action $(G:\co^{2m})$ has a reductive generic stabiliser of a special form, etc. This yields the equality
\[
    \max_{z\in\co^{2m}}\dim G{\cdot}z=\dim\CS_{2m}(\co)=2m\dim\co-2c(\co^m)-r(\co^m) .
\]
However, if $m>1$, then it is not easy to compute a $G$-generic stabiliser for $\co^{2m}$ and then 
the complexity and rank of $\co^m$.

\textbullet \ \ If $r=2m+1$, then a generic stabiliser for $(G:\co^{2m+1})$ is not necessarily reductive, and
results on doubled actions do not apply.

\begin{ex}    \label{ex:higher-omin}
Let $\co=\omin$ be the minimal nilpotent orbit in $\spn=\spv$. Then $\bco\simeq \eus V/\BZ_2$,
where the nontrivial element of $\BZ_2$ acts as $-1$ on $\eus V$. Therefore, a generic isotropy 
subalgebra for the action $(Sp(\eus V): \co^r)$ is the same as for the representation 
$(Sp(\eus V): r\eus V)$. By~\cite[Table\,2]{ag72}, \ 
$\mathsf{g.i.s.}(Sp(\eus V): r\eus V)=\begin{cases}  0, & \text{ if } r\ge 2n, \\
 \mathfrak{sp}_{2n-2m}, & \text{ if }  r=2m<2n.  \end{cases}$
\\
If $r=2m+1<2n$, then a generic stabiliser is the stabiliser of a nonzero vector in the space of 
tautological representation of $Sp_{2n-2m}$. This readily implies well-known results on $\CS_r(\omin)$ in the symplectic case, 
see e.g.~\cite[Theorem\,1.1]{badr}.
\end{ex}

Likewise, it $\co_1$ and $\co_2$ are different nilpotent orbits, then we may consider the {\it join\/} of their 
projectivisations, $\mathcal J(\ov{\BP\co_1}, \ov{\BP\co_2})\subset \BP\g$~\cite[Chap.\,5.1]{land}. For
the corresponding affine cone $\CJ(\co_1,\co_2)\subset \g$, one similarly obtains that
\[
     \dim \CJ(\co_1,\co_2)=\max_{z\in \co_1\times\co_2}\dim G{\cdot}z .
\]
However, this seems to be not very practical in general.

\appendix
\section{Determinantal varieties}  
\label{appendix}

\noindent
Here we gather some standard results on three types of determinantal varieties.

{\it\bfseries (1)} \ If $r\le N$, then
\[
    \sD_r(N)=\{M\in \glN \mid \rank(M)\le r\}
\]
is a {\it determinantal variety}. It is known that the variety 
$\sD_r(N)$ is irreducible and normal, and $\dim \sD_r(N)=r(2N-r)$, see~\cite[Chap.\,5.2]{smt}. Therefore
$\dim(\sD_r(N)\cap\slN)=r(2N-r)-1$. In this case, $\sD_r^o(N):=\sD_r(N)\cap\slN$ is also irreducible.

{\bf Remark.} Our variety $\sD_r(N)$ is denoted by $D_{r+1}(N)$ in book~\cite{smt}.

{\it\bfseries (2)} \  \ Let $\sk(N)$ denote the space of usual skew-symmetric matrices in $\glN$. If 
$2p\le N$, then
\[
        \sk_{2p}(N)=\{M\in \sk(N) \mid \rank(M)\le 2p\}
\]
is a {\it skew-symmetric determinantal variety}. The variety
$\sk_{2p}(N)$ is irreducible and normal, and $\dim\sk_{2p}(N)=p(2N-2p-1)$, see~\cite[Chap.\,7.2]{smt}. 
Let $\cI=\cI_N$ be the $N\times N$-matrix with 1's along the antidiagonal and zeros otherwise. Then 
$M^t=-M$ if and only if $\cI M$ is skew-symmetric with respect to the antidiagonal. Using the linear isomorphism 
$M\mapsto \cI M$, we may regard $\sk_{2p}(N)$ as variety of matrices that are skew-symmetric with 
respect to the antidiagonal. This is needed, if we regard $\soN$ as the space of skew-symmetric matrices w.r.t.~the antidiagonal.
        
{\it\bfseries (3)} \ Let $\sym(N)$ denote the set of all symmetric matrices of order $N$. If $m\le N$,
then
\[
     \sym_{m}(N)=\{M\in \sym(N) \mid \rank(M)\le m\}
\] 
is a {\it symmetric determinantal variety}. It is known that $\sym_{m}(N)$ is 
irreducible and normal, and $\dim\sym_{m}(N)=mN-\frac{m(m-1)}{2}$, see~\cite[Chap.\,6.2]{smt}. 
If $N=2n$ is even, then $\sym(2n)$ can be identified with $\spn$. For, if $\cJ$ is the matrix of an alternate
bilinear form that defines $\spn$, then \\
\centerline{
$M\in \spn$ \ if and only if \  $(\cJ M)^t=\cJ M$. }
\\[.6ex]
Using this linear isomorphism $M\mapsto \cJ M$, we may identify the symmetric determinantal varieties 
$\sym_{m}(2n)$ with subvarieties of $\spn$.

\end{document}